%% file: main.tex
\documentclass[reqno, 12pt]{amsproc}

\usepackage[utf8]{inputenc}
\usepackage{amsmath}
\usepackage[margin=1in]{geometry}
\usepackage{mathtools}
\usepackage{amssymb}
\usepackage{extarrows}
\usepackage{amsthm}
\usepackage{amsmath,mathabx}
\usepackage{mathtools}
\usepackage{cases}
\usepackage{booktabs}
\usepackage{blkarray}
\usepackage[english]{babel}
\usepackage{booktabs}
\usepackage{enumitem}
\usepackage{comment}
\usepackage{standalone}
\usepackage[linesnumbered,ruled,vlined]{algorithm2e}
\SetKwInput{KwInput}{Input}                
\SetKwInput{KwOutput}{Output}              
\usepackage[hidelinks,colorlinks=true,linkcolor=blue,citecolor=blue]{hyperref}
\usepackage[all,cmtip]{xy}
\usepackage{pgfplots,mathpazo}
\pgfplotsset{compat=1.17}
\usepackage{mathrsfs}
\usetikzlibrary{arrows,cd,patterns}

\theoremstyle{plain}
\newtheorem{theorem}{Theorem}[section]
\newtheorem{theoremlettered}{Theorem}

\newtheorem*{theorem*}{Theorem}
\newtheorem{prop}[theorem]{Proposition}
\newtheorem*{prop*}{Proposition}
\newtheorem{lemma}[theorem]{Lemma}
\newtheorem{cor}[theorem]{Corollary}

\newtheorem*{conj*}{Conjecture}
\theoremstyle{remark}
\newtheorem{rem}[theorem]{Remark}

\newtheorem{ex}[theorem]{Example}

\theoremstyle{definition}
\newtheorem{definition}[theorem]{Definition}
\newtheorem{nota}[theorem]{Notation}
\newtheorem*{nota*}{Notation}

\newcommand{\bbN}{\mathbb{N}}
\newcommand{\bbZ}{\mathbb{Z}}
\newcommand{\bbR}{\mathbb{R}}
\newcommand{\bbC}{\mathbb{C}}

\DeclareMathOperator{\rank}{rank}
\newcommand{\set}[1]{\left\{#1\right\}}
\newcommand{\Gfloor}[2]{\lfloor#1\rfloor_{#2}}
\DeclareMathOperator{\udim}{\underline\dim}
\DeclareMathOperator{\sign}{sign}
\newcommand{\K}{\mathbb{F}} 
\newcommand{\cl}[1]{[#1]} 
\newcommand{\Lan}[2]{\mathrm{Lan}_{#1}(#2)} 

\newcommand{\calZ}{\mathcal{Z}}
\newcommand{\Int}{\text{\bf I}}

\newcommand{\Vect}{\text{\rm Vect}}

\newcommand{\HN}[3]{\langle #3,#1\rangle^{#2}}
\newcommand{\subrep}[1]{\langle {#1} \rangle}

\newcommand{\HNtype}[2]{\mathbf{T}[#1;#2]}
\newcommand{\shift}[2]{\mathrm {sh}^{#1}_{#2}}

\newcommand{\id}{\text{id}}
\newcommand{\diff}{\mathrm d}

\newcommand{\Persfp}[1]{\mathrm{Pers}_\text{fp}(#1)}
\newcommand{\Pers}[1]{\mathrm{Pers}(#1)}

\newcommand{\ie}{\textit{i.e.}\ }
\DeclareMathOperator{\Img}{\mathrm{Im}}
\newcommand{\len}[1]{\mathrm{len}(#1)}
\newcommand{\Fil}[1]{\mathrm{Filtr}(#1)}
\newcommand{\ind}[1]{{\mathbb 1}_{#1}}
\newcommand{\vol}{\text{vol}}

\newcommand{\cub}[1]{\mathrm{cub}_{#1}}

\title{Harder-Narasimhan filtrations of persistence modules: metric stability}
\author{Marc FERSZTAND}

\begin{document}

\begin{abstract}
The Harder-Narasimhan types are a family of discrete isomorphism invariants for representations of finite quivers. Previously \cite{HN_discr}, we evaluated their discriminating power in the context of persistence modules over a finite poset, including multiparameter persistence modules (over a finite grid). In particular, we introduced the skyscraper invariant and proved it was strictly finer than the rank invariant. In order to study the stability of the  skyscraper invariant, we extend its definition from the finite to the infinite setting and consider multiparameter persistence modules over $\bbZ ^n$ and $\bbR^n$. We then establish an erosion-type stability result for this  version of the skyscraper invariant. 
\end{abstract}

\maketitle

\section*{Introduction}

\subsection*{Motivation}

Persistent Homology  is one of main tools in Topological Data Analysis (TDA). It provides a compact, computable and easy-to-read summary  measuring how certain geometric and topological features of the data persist across multiple scales. The main objects of study here are sequences of vector spaces connected by linear maps, usually called \emph{persistence modules}. These modules arise naturally when one computes the homology of a scale-indexed filtration of topological spaces built on an underlying dataset. Most Persistent Homology techniques output a discrete summary, called a {\em barcode}, which is complete in the sense that it characterises the given persistence module up to isomorphism \cite{edelsbrunner2000topological,zomorodian2004computing}. This barcode is known to be stable under certain natural perturbations of the persistence modules \cite{chazal2009proximity}.

For some applications, scale is not the only parameter of interest. For instance, adding density as a second parameter makes Persistent Homology techniques more robust to outliers in the data. It therefore becomes necessary to study more general notions of persistence modules, which are given by functors from arbitrary posets to the category of vector spaces. However, Carlsson and Zomorodian showed \cite{carlsson2009theory} that there are no discrete complete invariants for such persistence modules with two or more parameters. They also  introduced the \emph{rank invariant}, a (necessarily incomplete) discrete invariant  which associates to a multiparameter persistence module  the rank of all its maps.  Many other incomplete discrete invariants have since been proposed in the TDA literature. Examples of those include invariants using  ideas from homological  and commutative algebra \cite{bettis,host,blanchette2021homological}, sheaf theory \cite{ks,macpat},  lattice theory \cite{ kim2021generalized, botnanoudot,patel,asashiba2023approximation} and beyond  \cite{cerri2013betti}. More details about  invariants in multiparameter persistence, including the recent approach using relative homological algebra, can be found in the review articles \cite{botnan2022introduction} and \cite{blanchette2023exact}.

For every pointwise finite-dimensional (p.f.d) persistence module over a finite poset, one can define its  \emph{Harder-Narasimhan}(HN) \emph{filtration} \cite{king1994moduli,hille2002stable}. These filtrations are parameterised by a \emph{stability condition} which can be described by two real-valued functions on the elements of the poset. The dimensions of the vector spaces appearing in a  HN filtration are called a \emph{HN type}. They  can be used to build  the \emph{skyscraper invariant}  which is discrete, computable in polynomial time \cite{cheng2023deterministic} and  strictly stronger than the rank invariant  \cite{HN_discr}.

\subsection*{Stability} Multiparameter persistence modules indexed over $\bbZ^n$ or $\bbR^n$ are equipped with the \emph{interleaving distance} \cite{chazal2009proximity,lesnick2015theory} which makes standard TDA pipelines robust to noise in the data. When introducing an invariant of persistence modules, it is desirable to replicate the stability property enjoyed by barcodes of one parameter persistence modules. Explicitly, one seeks a metric on the output space which makes the invariant Lipschitz continuous with respect to the interleaving distance. For instance, the \emph{erosion distance} makes the rank invariant  stable \cite{patel2018generalized,puuska2020erosion,kim2021spatiotemporal} and a similar statement holds for the generalised rank invariant introduced by Kim and Mémoli  \cite{clause2022discriminating}. Moreover, rank exact decompositions, which are an example of the signed decompositions considered in \cite{blanchette2021homological,blanchette2023exact},  are  stable when equipped with a generalisation of the  bottleneck distance \cite{botnan2022bottleneck}.

In this work, we establish a stability result for the skyscraper invariant. In order to do so, we extend the definition of this invariant from the case of p.f.d. persistence modules over finite posets to the setting of finitely presentable persistence modules over $\bbZ^n$ or $\bbR^n$. Using the erosion distance, we define the HN distance on the output of this  version of the skyscraper invariant so that the rank invariant factors continuously through the skyscraper invariant.

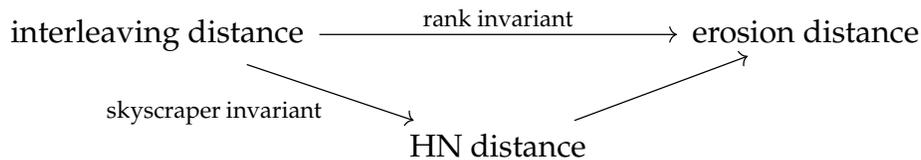
\[\begin{tikzcd}
  \text{interleaving distance}\ar[swap,dr,"\text{skyscraper invariant}"]\ar[rr,"\text{rank invariant}"]&\  &\text{erosion distance} \\
  \   & \text{HN distance} \ar[ru,""]& \ 
\end{tikzcd}\]

The Harder-Narasimhan filtrations  were introduced as stratifications of finite-rank vector bundles over a complex curve  and were used to compute the cohomology groups of certain moduli spaces \cite{Harder1974}. It was later adapted  to the case of finite quiver representations and hence of persistence modules over a finite poset \cite{king1994moduli, hille2002stable}. HN filtrations have also been studied in categories which are not necessarily abelian  \cite{bridgeland2007stability,andre2009slope}. However, when applied to the case of finitely presentable $\bbR^n$-persistence modules, these frameworks make restrictive assumptions  about the stability condition.

This paper together with \cite{HN_discr,fersztand2024harder} are part of an effort to use methods inspired by Geometric Invariant Theory in TDA.

\subsection*{Outline and main results}

Fix a  poset $Q$, the \emph{dimension vector} of a pointwise finite-dimensional $Q $-persistence module $V\in \Vect^Q$ is the function $\udim_V\colon x\in Q\mapsto \dim V_x\in \bbZ$. A \emph{stability condition} $Z$ over $Q$ is an additive function sending the dimension vector of a nonzero $Q$-persistence module to a complex number with strictly positive real part\footnote{$Z$ is a group homomorphism  from the Grothendieck group of  (an abelian subcategory of) $\Vect^Q$ to the additive group of complex numbers sending nonzero objects to the right half plane.}.  The $Z$-\emph{slope} of a nonzero  $Q$-persistence module $V\in \Vect^Q$ is the real number
\[\mu_Z(V):=\frac{\Im (Z(\udim_V))}{\Re (Z(\udim_V))}\]
 where $\Re$ and $\Im$ denote the real and imaginary parts. Moreover,  $V$ is said to be $Z$-\emph{semistable} if the $Z$-slopes of its nonzero submodules do not exceed $\mu_Z(V)$. If  $Q$ is finite, there is a unique finite-length filtration:
\[0=V^0\subsetneq V^1\subsetneq \dots \subsetneq V^\ell =V\]
called the \emph{HN filtration} \cite{Harder1974} of $V$ along $Z$ such that the successive quotients $S^i:=V^i/V^{i-1}$ are $Z$-semistable and have strictly decreasing $Z$-slopes. The \emph{HN type} of $V$ along $Z$ is the discrete invariant given by the dimension vectors of the quotients $S^1,\dots,S^\ell$. However, when $Q$ is infinite, Example \ref{ex:reasons_FG} shows that HN filtrations are not always well-defined. 

In Section \ref{sec:HN-exists}, we introduce two sufficient conditions on $Q$ and $Z$ to guarantee the existence of HN filtrations of finitely presentable $Q$-persistence  modules along $Z$.
 The first setting defines a continuous version of the skycraper weights  introduced in \cite{HN_discr} where  $\Im Z$ is built using the functions 
 \[\delta_q\colon \udim_V\mapsto \dim V_q\]
with $q\in Q$.  In the second setting, we restrict ourselves to the case where $Q=\bbR^n$ and we consider a large family of stability conditions induced by complex-valued functions on $Q$. More, precisely, given an integrable function $f\colon Q\to\bbC$ whose real part is strictly positive, we consider the stability condition $Z_f$ defined by
 \[Z_f\colon \udim_V\mapsto \int_{Q} f\udim_V.\]
We are particularly interested in the case where $f$ is a step function -- meaning that it is constant on the parts of a locally finite partition of $Q$ into hyperrectangles. Theorem \ref{thm:existence_HN} gives a more general version of the following statement:

\begin{theoremlettered}\label{thm:main-a}Let $Z$ be a stability condition over an upper semilattice $Q$. If either
\begin{enumerate}[label=(\roman*)]
    \item    $\Im Z$ is a nonnegative finite linear combination of the functions $(\delta_q)_{q\in Q}$; or,
    \item $Q=\bbR^n$ and $Z=Z_f$ is induced by an integrable step function $f\colon \bbR^n\to\bbC$ such that  $\Re f$ is strictly positive and  $\Im f$ is compactly supported,
\end{enumerate}
then HN filtrations along $Z$ of finitely presentable $Q$-persistence modules  exist, are unique and can be computed  over a finite poset.
\end{theoremlettered}

For the rest of this summary, $Z$ is a stability condition   over $Q=\bbR^n$ which satisfies the hypotheses of Theorem \ref{thm:main-a}. Let $x\in \bbR^n$, the $x$-\emph{shift} of  $V\in\Vect^Q$ is the $\bbR^n$-persistence module $T_x^*V$ obtained by shifting all the indices of $V$ by $x$.   Given a finitely presentable $\bbR^n$-persistence module $V$, temporarily denote by  $V^{x,\bullet}$  the HN filtration of $V$ along the stability condition  $Z^x\colon\udim_V\mapsto Z(\udim_{T_x^*V})$. For each $\theta\in\bbR$, we define (Definition \ref{def:rank filtrations}) the  function $s^\theta_{Z,V}$ which sends $x\leqslant y\in \bbR^n$ to the integer
 \[s_{Z,V}^\theta (x,y):=\rank(V^{x,i})_{x\leqslant y}\]
    where $i$ is the largest index such that either $i=0$ or $\mu_{Z^x}(V^{x,i}/V^{x,i-1})\geqslant \theta$. The  functions $(s^\theta_{Z,V})$ constitute a descending filtration of the rank invariant  indexed by  $\theta\in\bbR$. We call it the \emph{HN filtered rank invariant}  of $V$ along $Z$. 
   
   In Section \ref{sec:stability}, we establish a stability result for HN filtered rank invariants. Let $V$ and $W$ be two  $\bbR^n$-persistence modules and denote by $\vec\varepsilon$  the vector $(\varepsilon,\dots,\varepsilon)\in\bbR^n$. The \emph{interleaving distance}  $d_I(V,W)$ between $V$ and $W$ \cite{lesnick2015theory} is the infimum of all $\varepsilon \geqslant 0$ such that there exist maps of persistence modules 
   \[\phi_\bullet\colon V\to T_{\vec\varepsilon}^*W\hspace{3em}\text{ and }\hspace{3em}\psi_\bullet\colon W\to T_{\vec\varepsilon}^*V\]
   whose compositions $\psi_{\bullet+\vec\varepsilon}\circ\phi_\bullet$ and $\phi_{\bullet+\vec\varepsilon}\circ\psi_\bullet$ are given by the internal maps of $V$ and $W$. When $V$ is finitely presentable, we prove that each $s^\theta_{Z,V}$ can be seen as a functor from $(\bbR^n)^\text{opp}\times\bbR^n$ to $(\bbZ_+\cup\set{\infty})^\text{opp}$.  The \emph{erosion distance} $d_E(s,r)$  between two such functors $r$ and $s$ \cite{patel2018generalized} is the infimum of all $\varepsilon\geqslant 0$ such that for each $(x,y)\in \bbR^n \times\bbR^n$, we have
   \[s(x-\vec\varepsilon,y+\vec\varepsilon)\leqslant r(x,y)\qquad\text{and }\qquad r(x-\vec\varepsilon,y+\vec\varepsilon)\leqslant s(x,y).\]

   The following result, stated in full details in Theorem \ref{thm:stability_erosion}, shows that HN filtered rank invariants are stable:
\begin{theoremlettered}\label{thm:main-b}  For every finitely presentable $\bbR^n$-persistence modules $V$ and $W$, we have
\[\sup_{\theta \in\bbR}d_E(s_{Z,V}^\theta,s_{Z,W}^\theta)\leqslant d_I(V,W).\]
\end{theoremlettered}

The above  theorem may be used to produce a stability result for the HN filtered version of the persistence landscapes \cite{persistencelandscapes,vipond2020multiparameter} -- see Corollary \ref{cor:landscape-stability}.

Finally, in Appendix  \ref{append:semialgebraic}, we prove in the setting of  Theorem \ref{thm:main-a}(ii)  that if   $\Im f$ is nonnegative, then the HN filtered rank invariants along $Z_f$ are semialgebraically constructible functions in the sense of \cite{mccrory1997algebraically}.

\subsection*{Acknowledgements}
The author would like to thank Vidit Nanda and Ulrike Tillmann for their invaluable advice and support. He is also grateful to Emile Jacquard for many helpful discussions. The author is a  member of the Centre for Topological Data Analysis, funded by the EPSRC grant EP/R018472/1. 

\section{Preliminaries} 
We first introduce  discretisable persistence modules  over general posets \cite{botnan2020relative,chacholski2021realisations}  and  over finite products of totally ordered sets \cite{lesnick2015-2dpers-visu}.  We then briefly go through the concepts of stability conditions, slope stability, HN filtrations and HN types in  abelian categories \cite{bridgeland2007stability} and we define the pullback of a stability condition. Finally, in the case of persistence modules over a finite poset \cite{hille2002stable}, we  recall the definition of the skyscraper invariant introduced in \cite{HN_discr}.

\subsection{Persistence modules} Given a  field $\K$ and a poset $P$, a $P$-\emph{persistence module} is a functor $V:P\to \Vect_\K$ from $P$ to  the category of $\K$-vector spaces and linear maps.  The \emph{dimension vector} of a $P$-persistence module $V$ is the  function $\udim_V:P\to \bbN\cup\set\infty$ which sends $p\in P$ to $\dim V_p$. The \emph{rank invariant} of a persistence module $V$  is the function $\rho_V:P\times P\to \bbZ\cup\set{+\infty}$ given by
\[\rho_V(p,{p'}):=\begin{cases}
    \rank V_{p\leqslant {p'}}&\text{ if }p \leqslant {p'}\text{ in }P\\
     +\infty&\text{ otherwise}\\
\end{cases}\]

We  fix once and for all the field $\K$ and will henceforth omit it from the notation. We further ask $P$-persistence modules to be pointwise finite-dimensional -- meaning that the dimension vectors are integer-valued functions. The   $P$-persistence modules define an abelian category $\Pers P$ where morphisms are  natural transformations.

A subset $I$ of $P$ is called a \emph{spread} (or sometimes an interval) if it is:
\begin{itemize}[label=-]
    \item \underline{connected}: for all $u,v\in I$, there is a sequence  $(u=u_0,u_1,\dots,u_\ell=v)$ in $I$ such that $u_i$ and $u_{i+1}$ are comparable for each $i$; and
    \item \underline{convex}: $I$ contains the set $\set{p\in P\mid u\leqslant p \leqslant v}$ for each $u\leqslant v$ in $I$.
\end{itemize}
For instance, given $p\in P$, the set $\subrep p:=\set{{p'}\in P\mid {p'}\geqslant p}$ is a spread.  The \emph{spread module} $\K_I$ is the indecomposable module whose spaces $(\K_I)_p$ are $\K$ if $p\in I$ and 0 otherwise and whose maps between   nonzero spaces are  identities.

\subsection{Discretisable persistence modules} Let $P$ and $Q$ be two posets. A \emph{map of poset} $f\colon P\to Q$ is an increasing map from $P$ to $Q$. The \emph{pullback} of a $Q$-persistence module $V$ by $f$ is the $P$-persistence module
\[f^*V:=V\circ f\]
where $\circ$ is the composition between functors. The \emph{pushforward} of a $P$-persistence module $U$ by   $f$ is the $Q$-persistence module 
\[f_*U:=\Lan f U\]
where $\Lan f U$ is the left Kan extension  of $U$ along $f$. More concretely, for $q\in Q$, we have $(f_*U)_q:=\underset{p\mid f(p)\leqslant q}\varinjlim U_p$. The functors $f^*\colon \Pers Q\to \Pers P$ and $f_*\colon \Pers P\to \Pers Q$ are respectively called the restriction (sometimes pullback) and induction (sometimes extension \cite{blanchette2023exact} or pushforward) functors. They form \cite[Lemma 2.15]{botnan2020relative} an adjoint pair $(f_*,f^*)$ and if moreover $f$ is  an inclusion of posets (\ie $f$ is fully faithful), then the unit $\id_{\Pers P}\Rightarrow f^*f_*$ is a natural isomorphism \cite[Lemma 2.19]{botnan2020relative}.

\begin{definition}\label{def:discretisable}
    A $Q$-persistence module $V$ is said to be $f$-\emph{discretisable} if it is isomorphic in $\Pers Q$ to the pushforward by $f$ of some $P$-persistence module $U$
    \[V\simeq f_*U.\]
     In that case, we say that $f$ is \emph{adapted} to $V$.  A $Q$-persistence module is said to be \emph{discretisable} (or tame) if  it is $f$-discretisable for some map of posets $f$ from a \underline{finite} poset $P$ to $Q$.   We denote by $\Persfp Q\subset \Pers Q$  the full subcategory  of discretisable $Q$-persistence modules.
\end{definition}
 Observe that the induction functor of a map of (not necessarily finite) posets $g\colon Q\to Q'$ sends tame $Q$-persistence modules to tame $Q'$-persistence modules.
\begin{ex}
\begin{itemize}
    \item When $Q$ is finite, we have $\Persfp Q=\Pers Q$
    \item Discretisable $\bbR$-persistence modules are the ones whose barcode consists of a finite number of intervals of the form $[a,b)$ with $(a,b)\in \bbR\times \bbR\cup\set\infty$
\end{itemize}
    \end{ex}

    A $Q$-persistence module is \emph{free and finitely generated} if it decomposes as a direct sum of a finite number of spread modules of the form $\K_{\subrep q}$ with $q\in Q$. A $Q$-persistence module $V$ is said to be \emph{finitely presentable} if there is a map $f$ between two free and finitely generated $Q$-persistence modules such that $V\simeq\mathrm{coker\ } f $. Our interest in finitely presentable persistence modules is justified by  the fact that a $Q$-persistence module lies in $\Persfp Q$ if and only if it is finitely presentable {\cite[10.3]{chacholski2021realisations}}. 

   A \emph{refinement} of $f\colon P\to Q$ is a map of posets $f'\colon P'\to Q$ such that $f$ factors through $f'$. Namely, there is a map of posets $g$ fitting into the commutative diagram
   \[\begin{tikzcd}
       P\ar[d,"f",swap]\ar[r,dotted,"g"]&P'\ar[dl,"f'"]\\
       Q
   \end{tikzcd}\]
   \begin{rem}\label{rem:adapted_to_incl}
         For every subset $S$ of $Q$ containing $f(P)$, the inclusion of posets $f'\colon S\hookrightarrow Q$ is a refinement of $f$. Hence, a pair of maps of posets $f^1\colon P^1\to Q$ and $f^2\colon P^2\to Q$ has a common refinement given by the inclusion $f'\colon f^1(P^1)\cup f^2(P^2)\hookrightarrow Q$. 
   \end{rem}

   The existence of common refinements implies that the subcategory  $\Persfp Q$ is additive; however, in general, it is not abelian. It is thus necessary to make assumptions about the poset $Q$. We recall that $Q$ is an \emph{upper semilattice} if every pair of elements of $Q$ admits a least upper bound. 

   \begin{lemma}
       \label{lmm:abelianity-semilattice}
       Assume that $Q$ is an upper semilattice, then,
       \begin{enumerate}[label=(\roman*)]
           \item  $\Persfp Q$ is an abelian  full subcategory of $\Pers Q$ {\cite[8.5]{chacholski2021realisations}} and 
           \item  any map of posets $f\colon P\to Q$ from a finite upper semilattice $P$  has an exact induction functor  $f_*\colon \Pers {P}\to  \Pers Q$ \cite[10.5]{chacholski2021realisations}. 
       \end{enumerate}
   \end{lemma}

We now assume that $Q$ is an upper semilattice. By remark \ref{rem:adapted_to_incl}, every discretisable $Q$-persistence module can be discretised by an inclusion  of a finite upper semilattice $f\colon P\hookrightarrow Q$ into $Q$. In this setting,    we can define the $f$-\emph{floor function} which sends $q\in Q$ to
\[\Gfloor{q}{f}:= \max \set{p\in P\mid f(p)\leqslant q}\]
where $\Gfloor{q}{f}\in P\cup\set{-\infty}$. By definition, the pushforward of some $U\in \Pers P$ by $f$ can be written pointwise as $(f_*U)_q= U_{\Gfloor{q}{f}}$ with the convention $U_{-\infty}=0$. We further denote
\[ \cub f(p) :=\set{q\in Q\mid \Gfloor {q}f=p}\]
the fiber of $p\in P\cup\set{-\infty}$ by $\Gfloor{\ }f$.  

We will reuse the vocabulary defined in Definition \ref{def:discretisable}  for functions from $Q$. Given a map of upper semilattices $f\colon P\hookrightarrow Q$ with $P$ finite and a set $Y$, the restriction functor $f^*$ is the precomposition by $f$ and the induction functor $f_*$ sends a function $a\colon P\to Y$  to 
\[f_*a(q):= a(\Gfloor{q}{f})\]
where $q\in Q$.  Let $\mathcal F(Q,Y)$ denote the set of  functions $a\colon Q\to Y$ which are discretisable by some map of posets $f\colon P\hookrightarrow Q$ with $P$ finite. Discretisable functions over a semialgebraic set are semialgebraically constructible in the sense of Definition \ref{def:constructible}. Observe that for $U\in \Pers P$ and $V\in\Pers Q$, we have 
\begin{equation}\label{eqn:dim_commute_pushforward}
    \udim_{f^*V}=f^*\udim_V\qquad\text{and}\qquad \udim_{f_*U}=f_*\udim_U.
\end{equation}
Hence, the assignment $\udim\colon V \mapsto \udim_V$ defines a  surjective function from $\Persfp Q$ to $ \mathcal F(Q,Y)$.

\subsection{Grid functions} In this subsection, the poset $Q$ is the product of $n$   totally ordered posets equipped with the product partial order. This setting contains our two main cases of interest  where $Q$ is either $\bbZ^n$ or $\bbR^n$. We will give a more concrete definition of discretisable persistence modules.

\begin{definition}[{\cite[Section 2.5]{lesnick2015-2dpers-visu}}]\label{def:grid_fun}
    A \emph{grid function}  $G:=G_1\times\dots\times G_n$ over $Q$ is an embedding of a  cube $P=P_1\times \dots \times P_n$ of $ \bbZ^n$ into $Q$ which is defined coordinate-wise by  strictly increasing functions $G_i\colon P_i\hookrightarrow Q_i$ and such that $\lim_{t \to \pm \infty} G_i(t)=\pm\infty$ -- if those limits exist. When $P$ is finite,  $G$ is said to be a finite grid function.
\end{definition}

Observe that by Remark \ref{rem:adapted_to_incl}, a $Q$-persistence module is discretisable if and only if it is discretisable by a finite grid function.

An \emph{interval} in a totally ordered poset $T$ is a subset $I\subset T$ such that for $x,y\in I$, we have $x\leqslant z\leqslant y\implies z\in I$. A \emph{cube} of $Q=Q_1\times \dots \times Q_n$ is a subset $I_1\times \dots\times I_n\subset Q$ such that each $I_i$ is an interval of $Q_i$. 

The interest of working with a finite grid function $G\colon P\hookrightarrow Q$ is that the floor function $\Gfloor{\ }G$ can be computed coordinate-wise and that for every $p\in P$  the fibers $\cub G({p})$ are cubes of the form $\prod_{i=1}^n[a_i,b_i)$ where  $(a_i,b_i)\in Q_i\times (Q_i\cup\set\infty)$. In the example below, the black disks are the image of $G$, the dashed area is a fibre  of some $p\in P$ under $\Gfloor\ G$ and the square node is some $q\in Q$:

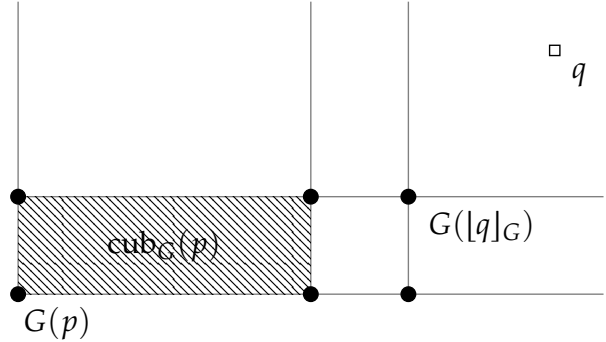
\begin{center}
    \includestandalone[width=8cm]{images/grid2}
\end{center}

Let $V$ be  a $G$-discretisable $Q$-persistence modules. For $q\in Q$ such that $\Gfloor qG\neq -\infty$, the map $V_{G(\Gfloor qG)\leqslant q}$ is an  isomorphism. In particular, the dimension vector $\udim_V$ is constant in each of the $\cub G(p)$.

Most of the above notions can  easily be adapted to the case where $G\colon P\hookrightarrow Q$ is an infinite grid function. Indeed, for $q\in Q$, the subset $\set{p\in P\mid G(p)\leqslant q}$ is a cube of $\bbZ^n$ which is bounded above in each direction so that  its maximum $\Gfloor qf$ is well-defined. Hence, as before, the induction functor $G_*$ is  exact. 

Let $G\colon P\hookrightarrow Q$ be a grid function, a refinement $G'$ of $G$ such that $\Img G=\Img G'\cap c$ for some cube $c$ of $Q$ is called an \emph{extension} of $G$. 

  \begin{rem}\label{rmk:noninj-grid}
For convenience, we will sometimes consider product of $n$ maps of posets $G\colon P\to  Q$ which are not necessarily injective. We call these grid functions \emph{improper} grid functions.
    \end{rem}

\subsection{HN filtrations in abelian categories}\label{subsec:HN_abelian}

Let $\mathcal A$ be an  abelian category. The \emph{Grothen\-dieck group} of $\mathcal A$ is the abelian group $K(\mathcal A)$  generated by the isomorphism classes $\cl V$ of objects subject to the relations $\cl {V'}=\cl V+\cl {V''}$ for every short exact sequence $0\to V\to V'\to V''\to 0$ in $\mathcal A$. A \emph{stability condition} $Z$ over $\mathcal A$ is a group  morphism  between the Grothendieck group of $\mathcal A$ and  $(\bbC,+)$ which sends nonzero objects to the right open half space $\set{z\in\bbC\mid\Re \  z>0}$. The \emph{slope} of a nonzero object $V\in\mathcal A$ along  $Z$ is the real number $\mu_{Z}(V):=\frac{\Im Z(\cl V)}{\Re  Z(\cl V)}$. Moreover, $V$ is said to be $Z$-\emph{stable} if all its proper subobjects  have a strictly smaller  slope:
\[ 0\subsetneq W\subsetneq V \quad\Longrightarrow\quad \mu_Z(W)<\mu_Z(V)\]
 if instead the inequalities are weak, then $V$ is said to be $Z$-\emph{semistable}. 

\begin{definition}[{\cite{Harder1974,hille2002stable,rudakov}}]\label{def:HN_abelian} Given a nonzero object $V$ in $\mathcal A$, a HN filtration of $V$ along $Z$  is a finite sequence of subobjects $0=V^0\subsetneq V^1\subsetneq \dots\subsetneq V^\ell=V$ such that each quotient $V^{j+1}/V^j$ is $Z$-semistable and $\mu_Z(V^1/V^0)>\mu_Z(V^2/V^1)>\dots>\mu_Z(V^\ell/V^{\ell-1})$.
\end{definition}

\begin{prop}\label{prop:exist_HN_abelian}{\cite[Proposition 2.4]{bridgeland2007stability}}  When they exist, HN filtrations are unique. Moreover, if  the following two conditions are satisfied 
\begin{enumerate}
    \item there are no infinite sequences of subobjects 
$\dots\subset V_{j+1}\subset V_j \subset\dots\subset V_2\subset V_1$ in $\mathcal A$  with $\mu_Z(V_{j+1})>\mu_Z(V_j)$  for all $j$,
    \item there are no infinite sequences of quotients 
$V_1\twoheadrightarrow V_2 \twoheadrightarrow \dots \twoheadrightarrow V_j\twoheadrightarrow V_{j+1}\twoheadrightarrow\dots$ in $\mathcal A$ with $\mu_Z(V_{j})>\mu_Z(V_{j+1})$  for all $j$,
\end{enumerate}
then every   nonzero object in $\mathcal A$ has a unique HN filtration along $Z$.    
\end{prop}

\begin{nota}
    \label{nota:HNfiltr}
 If $0=V^0\subsetneq V^1\subsetneq\dots\subsetneq V^\ell=V$ is the HN filtration  of $V$  along  $Z$, we use the notation $\HN Z\theta V=V^i$ whenever $\mu_Z(\frac{V^i}{V^{i-1}} )\geqslant\theta > \mu_Z(\frac{V^{i+1}}{V^{i}})$ with the conventions $\mu_Z(\frac{V^0}{V^{-1}})=+\infty$ and $\mu_Z(\frac{V^{\ell+1}}{V^\ell})=-\infty$. For convenience, we write $\HN Z\theta 0=0$ for all $\theta$.
\end{nota}

Let $\Fil {\mathcal A}$ be the category of $\bbR^\text{opp}$-indexed filtrations of objects of $\mathcal A$. Given a stability condition $Z$, we refer to   the assignment $\HN{Z}{}{-}\colon V\in \mathcal A\mapsto(\HN{Z}{\theta}{V})_{\theta\in\bbR^\text{opp}}\in \Fil{\mathcal A}$ as the \emph{HN functor} along $Z$. Indeed, it has been observed in many settings that under some finiteness assumptions on $Z$ and $\mathcal A$, the assignment   $ \HN{Z}{}{-}$ is well-defined and functorial  -- see \cite{faltings1995mumford}, \cite[Theorem 2.8]{hille2002stable} and \cite[3.1]{andre2009slope}. 

\begin{definition}Assume that a nonzero object $V\in \mathcal A$ has a HN filtration along $Z$. Then its \emph{HN type} along $Z$ is the data of the dimension vectors appearing in its HN filtration. Using Notation \ref{nota:HNfiltr}, we write it 
    \[\HNtype{V}{Z}:=\left(\udim_{\HN{Z}{\theta}V}\right)_{\theta\in\bbR^\text{opp}}.\]
\end{definition}

\subsection{Pullback of stability conditions}\label{subsec:pullback-Z}
Let $F\colon \mathcal A\to \mathcal A'$ be an additive exact functor between abelian categories and let $Z\colon K(\mathcal A')\to\bbC$ be a stability condition over $\mathcal A'$. We further assume that the group homomorphism $K(   \mathcal A) \to K(\mathcal A') $ induced by $F$ is injective. Then, one can define the \emph{pullback} $F^*Z\colon K(\mathcal A)\to\bbC$ of $Z$ by $F$ to be the stability condition over $\mathcal A$ which fits into the following commutative diagram:
\begin{equation}\label{eqn:commute_Z}
    \begin{tikzcd}
K(   \mathcal A) \ar[r,hookrightarrow,"F"]    \ar[dotted,swap,dr,"F^*Z"] & K(\mathcal A') \ar[d,"Z"] \\ 
   & \bbC
\end{tikzcd}
\end{equation}

\begin{lemma}\label{lmm:converse_ss_conserved}
    For every  nonzero object $U$  of $ \mathcal A$, one has $\mu_{F^*Z}(U)= \mu_{Z}(F(U))$ and moreover,
\[F(U)\text{ is $Z$-semistable}\implies U\text{ is ${F^*Z}$-semistable}.\]
\end{lemma}
\begin{proof}
    The equality $\mu_{F^*Z}(U)= \mu_{Z}(F(U))$ is a direct consequence of \eqref{eqn:commute_Z}. Assume now that $F(U)$ is $\mu_Z$-semistable and let $0\subsetneq U'\subsetneq U$. By exactness of $F$, we have $F(U')\subset F(U)$ and by semistability $\mu_{F^*Z}(U')=\mu_{Z}(F(U'))\leqslant \mu_Z(F(U))=\mu_{F^*Z}(U).$
\end{proof}

The converse does not always hold, but if it does, $F$ commutes with the HN functor:

\begin{lemma}\label{lmm:preserve-ss-implies-HN} Assume  that  every nonzero $U\in\mathcal A$ has a HN filtration along $ F^*Z$ and satisfies:
\[ U \text{ is }F^*Z\text{-semistable }  \ \Longrightarrow \   F(U) \text{ is }Z\text{-semistable}.\]
    Then  every nonzero object $U\in \mathcal A$ has a HN filtration along $Z$ and  we have:
    \[ \HN{Z}{}{F(U)}=F(\HN{F^*Z}{}{U}).\]
    \end{lemma}
\begin{proof}
Let $U\in\mathcal A\setminus\set 0$ and let $0=U^0\subsetneq U^1\subsetneq \dots\subsetneq U^\ell=U$ be its HN filtration along $F^*Z$.  By Lemma \ref{lmm:converse_ss_conserved}, the modules $( F(U^{i}/U^{i-1}))_{1\leqslant i\leqslant  \ell}$ have decreasing $Z$-slopes and by assumption they are $Z$-semistable. The exactness of $F$ and the uniqueness of  HN filtrations ensure that $0= F(U^0)\subsetneq F(U^1)\subsetneq \dots\subsetneq  F(U^\ell)= F(U)$ is   the HN filtration of $ F(U)$ along $Z$. 
\end{proof}

\subsection{HN filtrations of persistence modules over finite posets}\label{subsec:HN-finite-poset} 
We fix a finite poset  $P$ and consider the abelian category $\Pers P$. Then, every stability condition satisfies the assumptions of Proposition \ref{prop:exist_HN_abelian} and is of the form
\begin{equation}\label{eqn:stab_cond_quiver}  Z_{\alpha,\beta}(\cl V)=\sum_{p\in P}(\beta(p) +\sqrt{-1}\alpha(p) )\dim V_p\end{equation}
for some $\alpha\colon P\to \bbR$ and $\beta\colon P\to \bbR_{>0}$.
\begin{rem}\label{rem:quiver_poset}
    The finite poset $P$ can be seen as a quiver $\mathrm{Hasse}(P)$ whose vertices are the elements of $P$ and whose edges are given by the covering relation in $P$. Namely, $\mathrm{Hasse}(P)$ is the  directed simple graph such that given $p,{p'}\in P$, there is an edge from $p$ to ${p'}$ if and only if $p< {p'}$ and there is no ${p''}\in P$ such that $p<{p''}<{p'}$. 

    Let $V$ be a quiver representation and for a path $\gamma$ in the quiver let $V_\gamma$ denote the  composition of the maps of $V$ along $\gamma$.  The representation $V$ is said to be \emph{equalised} if for every pair of paths $\gamma,\gamma'$ with the same sources and targets, we have $V_\gamma=V_{\gamma'}$.

    Then, the $P$-persistence modules can be identified with the equalised representations of $\mathrm{Hasse}(P)$. These commutativity relations are conserved by taking subrepresentations. By Lemma \ref{lmm:preserve-ss-implies-HN}, it is equivalent to compute the HN filtrations of $V\in \Pers P$ as a $P$-persistence module and as a  representation of the quiver without relation $\mathrm{Hasse}(P)$. 

\end{rem}

Let $p\in P$, the \emph{skyscraper weight} at $p$ is the function $\delta_p\colon P\to \bbR$ which sends $p$ to 1 and ${p'}\neq p$ to $0$. We fix a  function $\beta\colon P\to \bbR_{>0}$, the \emph{skyscraper invariant} of $V\in \Pers P\setminus\set 0$ is the data of 
\begin{equation}
    \label{eqn:skyscraper}(\HNtype{V}{Z_{\delta_p,\beta}})_{p\in P}.
\end{equation}

Note that replacing $\HNtype{V}{Z_{\delta_p,\beta}}$ by $(\rho_{\HN{Z}{\theta}{V}})_\theta$ in \eqref{eqn:skyscraper} does not create a more informative invariant. Indeed, for every $\theta>0$, the maps between nonzero spaces of $\HN{Z}{\theta}{V}$ are all surjective    \cite[Proposition 3.3]{HN_discr}.

\section{HN filtrations of  persistence modules}%
\label{sec:HN-exists}

The two main goals of this Section are firstly, to find a framework where HN filtrations of nonzero discretisable $\bbR^n$-persistence modules exist, and secondly, to  define a continuous version of  the skyscraper invariant. 

We first study stability conditions of discretisable persistence modules over an upper semilattice $Q$. Since HN  filtrations do not always exist, even when $Q=\bbZ$ or $\bbR$ -- see Example \ref{ex:reasons_FG} -- we restrict ourselves to the family $\calZ^\text{eval}(Q)$ of  stability conditions whose imaginary part is induced by the Dirac delta functions. Given $Z\in\calZ^\text{eval}(Q)$,  we prove that every discretisable $Q$-persistence module  has a HN filtration along $Z$ which can be computed in any fine enough discretisation. 

Focusing on the case where $Q=\bbZ^n$ or $\bbR^n$, we consider a larger family $\calZ(Q)$ containing all the stability conditions induced by  step functions and we establish the existence of HN filtrations along any $Z\in \calZ(Q)$. More precisely, we show that the HN functor $\HN{Z}{}{-}$ along $Z\in\calZ(Q)$ is well-defined and  commutes with the induction functor $G_*$ of a fine enough finite or infinite grid function $G$. Finally,  we define in this setting the \emph{HN filtered rank invariant} along $Z$ which generalises and gives a continuous version of  the skyscraper invariant defined in \eqref{eqn:skyscraper}. 

\subsection{Stability conditions for persistence modules}%
    Let $Q$ be an upper semilattice.

\begin{lemma}
    The assignment $\udim\colon V\mapsto \udim_V $ induces an isomorphism from the Grothen\-dieck group  of discretisable $Q$-persistence modules $K(\Persfp Q)$ to the  discretisable functions  $\mathcal F(Q,\bbZ)$ from $Q$ to $\bbZ$. 
\end{lemma}
\begin{proof}
    If $0\to V\to V'\to V''\to 0$ is exact in $\Persfp Q$, then $\udim_{V'}-\udim_V-\udim_{V''}=0$ so by \eqref{eqn:dim_commute_pushforward} the homomorphism $\udim\colon K( \Persfp Q)\to \mathcal F(Q,\bbZ)$ is well-defined and surjective. Moreover, when $Q$ is finite, this homomorphism is clearly injective. Assume now that $Q$ is infinite and choose $ V\in K(\Persfp Q)$ in the kernel of $\udim$. Then, by  Remark \ref{rem:adapted_to_incl} and  Lemma \ref{lmm:abelianity-semilattice}(ii), there is an inclusion of upper semilattices  $f\colon P\hookrightarrow Q$ with $
    P$ finite and $ U\in K(\Persfp P)$ such that $ V=f_* U$. So, by Equation \eqref{eqn:dim_commute_pushforward}, we have
    \[0=f^*\udim ( V)=\udim(f^*f_* U)=\udim( U). \]
   By the finite case, $ U=0$ and finally $ V=f_*U=0$ in $K(\Persfp Q)$.
\end{proof}

 Hence, any stability condition $Z$   factors through $\udim$, and is uniquely determined by a linear form $\alpha$ and a strictly positive linear form $\beta$ from the real vector space $ \mathcal F(Q,\bbZ)\otimes \bbR\cong \mathcal F(Q,\bbR)$. More precisely, any stability condition is of the form 
\begin{equation}\label{eq:stab_cond_gen}
    Z_{\alpha,\beta}\colon \cl V\mapsto \beta(\udim_V)+\sqrt{-1}\alpha(\udim_V).
\end{equation}

Note that working in $\Persfp Q$ and not $\Pers Q$ is essential to guarantee the existence of a strictly positive linear form $\beta$ over the dimension vectors of $\bbR^n$-persistence modules. Henceforth, we use the above notation $Z_{\alpha,\beta}$ to denote the stability condition induced by the linear forms $\alpha$ and $\beta$.

Let $f\colon P\hookrightarrow Q$ be an inclusion of  upper semilattices. The induction $f_*$ is exact and induces an injective group homomorphism $K(\Persfp{ P})\to K(\Persfp{Q})$. We denote by $f^*Z$ the pullback  of $Z$ by $f_*$ which is defined by diagram \eqref{eqn:commute_Z} with $F=f_*$. 

Moreover, when $P$ is finite  and $Z:=Z_{\alpha,\beta}$ is expressed like in \eqref{eq:stab_cond_gen}, we write $ f^*Z:=Z_{f^*\alpha,f^*\beta}$ like in \eqref{eqn:stab_cond_quiver}. The functions $f^*\alpha,f^*\beta\colon  P\to \bbR$ are defined for $p\in P$ by
\[f^*\alpha(p):= \alpha(\ind{\cub f(p)}) \qquad\text{ and }\qquad f^*\beta(p):=\beta(\ind{\cub f(p)})\]
where $\ind{\cub f(p)}\in\mathcal F( Q,\bbZ)$ denotes the indicator function of $\cub f(p)\subset Q$.

The notions of slope, (semi)stability and HN filtrations are defined in $\mathcal A:=\Persfp {Q}$ as per Subsection \ref{subsec:HN_abelian}. Importantly, if we restrict ourselves to $f$-discretisable $Q$-persistence modules where $f$ is the inclusion of a finite upper semilattice $P$ into $Q$, then all the slopes can be computed in $\Pers P$ using the same framework as in  Subsection \ref{subsec:HN-finite-poset}.

\subsection{Existence of HN filtrations of persistence modules} The following examples in the single-parameter setting illustrate some reasons why HN filtrations of persistence modules do not always exist: 

\begin{ex}\label{ex:reasons_FG} \
 \begin{enumerate} 
    \item   Consider a stability condition $Z$ over $\bbR$ of the form 
\[Z\colon \cl V \in K(\Persfp \bbR)\mapsto  \int_{\bbR} (b+\sqrt{-1}a) \udim_V\in\bbC\]
where $a,b\colon\bbR\to\bbR_{>0}$ are integrable, $a_{|[0,1]}$ is increasing and $b_{|[0,1]}$ is constant. Then the set of $Z$-slopes attained by  the submodules of $\Int[0,1)$ does not have a maximum. More precisely, the submodules  of $\Int[0,1)$ can be ordered into  an infinite decreasing family $(\Int [t,1))_{t\in[0,1)}$ in $\Persfp {\bbR}$ with increasing $Z$-slopes. The module $\Int [0,1)$ has no $Z$-semistable submodule and hence no HN filtration. 
\item[(1')] A similar phenomenon occurs with  a stability condition $Z=Z_{\alpha,\beta}$  over $\bbR$ whose imaginary part is induced by $\alpha\colon f\in \mathcal F(\bbR,\bbR)\mapsto -f(0)\in\bbR$. The quotients of $\Int[0,1)$ can be ordered into an infinite increasing family $(\Int [0,t))_{t\in(0,1]}$  with increasing $Z$-slopes. Since the filtration   $0\subsetneq \Int(0,1)\subsetneq \Int[0,1)$ does not lie in $\Persfp \bbR$, the spread module $\Int[0,1)$ has no HN filtration along $Z$ in $\Persfp \bbR$. 

Note that   an abelian subcategory of $\Pers {\bbR}$ which contains  the inclusions  $\Int(p,p')\subset \Int[p,p')$   for  $p<p'\in\bbR$, would also contain $\Int [p,p')/\Int(p,p')=\Int[p,p]$, and the condition $  \beta(\K_{\set p})>0$ would exclude many intuitive choices for $\beta$.
\item Let $Z=Z_{\alpha,\beta}$ be a stability condition over $\bbZ$ whose imaginary part is induced by $\alpha\colon f\in \mathcal F(\bbZ,\bbR)\mapsto \lim_{+\infty} f\in\bbR$. The spread module $\Int[0,+\infty)$ has no HN filtration along $Z$ because its submodules can be ordered into an infinite decreasing family $(\Int[t,+\infty))_{t\in\bbZ_{\geqslant0}}$ of increasing $Z$-slope.
\end{enumerate}
\end{ex}

We will now consider certain stability conditions such that the induction functor preserves semistability. We will then use Lemma \ref{lmm:preserve-ss-implies-HN}  to guarantee the existence of HN filtrations.

Let $Q$ be an upper semilattice. The \emph{evaluation form} at $q\in Q$ is the linear form
\[\delta_{ q}:f\in\mathcal F(Q,\bbR)\longmapsto f(q)\in\bbR.\]
Given an inclusion $f$ of upper semilattices, we denote by $\calZ^\text{eval}(f)$ the set of stability conditions $Z:=Z_{\alpha,\beta}$ over $Q$ such that  $\alpha$ is a nonnegative linear combination of the evaluation forms $(\delta_q)_{q\in \Img f}$. If $Z\in\calZ^\text{eval}(f)$, we say that $f$ is \emph{adapted} to $Z$ or equivalently that $Z$ is $f$-\emph{discretisable}. We denote by $\calZ^\text{eval}(Q)$ the union of $\calZ^\text{eval}(f)$ over all inclusions $f $ of a finite semilattice into $Q$. 

\begin{prop}\label{prop:HN_existence_eval}
Let $f$ be the inclusion of a finite upper semilattice $P$ into $Q$ and assume that $Z\in\calZ^\text{eval}(f)$. Then, for any nonzero $P$-persistence module $U$, we have 
    \[ U \text{ is }f^*Z\text{-semistable }  \ \Longrightarrow \  f_*U \text{ is }Z\text{-semistable}.\]
\end{prop}
\begin{proof}Let $W$ be a submodule of $f_*U$, we first prove that the pixelisation $f_*f^*W$ is a submodule of $W$.  For $q\in Q$, consider the  linear map $\phi_q:=W_{f(\Gfloor qf)\leqslant q}$. Its domain is $(f^*W)_{\Gfloor qf} =(f_*f^*W)_q$ and it is injective as it is the restriction of the isomorphism $(f_*U)_{f(\Gfloor qf)\leqslant q}$. These maps induce a map of $Q$-persistence modules $\phi_\bullet\colon f_*f^*W\hookrightarrow W$ because for $q\leqslant q'$ in $Q$, we have,
\[\phi_{q'}\circ (f_*f^*W)_{q \leqslant {q'}}=\phi_{q'}\circ W_{f(\Gfloor qf) \leqslant f(\Gfloor {q'}f)}=W_{f(\Gfloor qf) \leqslant {q'}}= W_{q\leqslant {q'}}\circ \phi_{ q}.\]

   Assume now that $U\in \Persfp P$ is $f^*Z$-semistable and let $0\neq W\subset f_*U$. We prove that $\mu_Z(W)\leqslant \mu_Z(U)$. Since $\mu_Z(V)\geqslant0$, we can restrict ourselves to the nontrivial case where $\alpha(\udim_W)>0$.   Given $p\in P$, we have $\delta_{{f(p)}}(\udim_{f_*f^*W})=\dim(f^*W)_p=\delta_{{f(p)}}(\udim_W)$, whence $\alpha(\udim_{f_*f^*W})=\alpha(\udim_W)>0$. In particular, $f_*f^*W$ is nonzero.  Since $f_*f^*W$ is a submodule of $W$, we have $\udim_{f_*f^*W}\leqslant \udim_{W}$ and  $\mu_{Z}(f_*f^*W)\geqslant \mu_Z(W)$ because  $ \beta$ is positive. But, by applying the restriction functor, $f^*W\subset f^*f_*U\simeq U$ so by semistability of $U$ and Lemma \ref{lmm:converse_ss_conserved}, we finally have
   \[\mu_Z(W)\leqslant \mu_Z(f_*f^*W)=\mu_{f^*Z}(f^*W) \leqslant\mu_{f^*Z}(U)=\mu_Z(f_*U). \]
\end{proof}

We now restrict ourselves to the case where $Q$ is a cube of either $\bbR^n$ or $\bbZ^n$ and state an analog of Proposition \ref{prop:HN_existence_eval} for a larger family of stability conditions.

\begin{definition}\label{def:conditions-Z-discr}
Let $G$ be a grid function over $Q$. We denote by $\calZ^\text{step}(G)$ the set of stability conditions $Z\colon K(\Persfp Q)\to\bbC$ of the form
    \[
    Z([V]) = \int_Q (b+\sqrt{-1}a) \udim_V 
    \]
where $a,b\colon Q\to \bbR$ are  discretisable by an extension of $G$ and  
\begin{itemize}
    \item  $a$ is null on all the unbounded cubes $\cub G(x) $ delimited by $G$
    \item  $b$ is strictly positive and integrable on $Q$
\end{itemize}
\end{definition}

Let $c$ be a nonempty bounded cube of $Q$. The dimension  of $c$ is the number $0\leqslant\dim c\leqslant n$ of coordinates over which the projection of $c$ has nonempty interior. We define the positive linear form $\delta_c$ which sends $f\in\mathcal F(Q,\bbR)$ to its average over $c$. Namely, 
\begin{equation}\label{eqn:deltac}\delta_c(f):=\frac{\int_cf}{\int_c1}\end{equation}
where the integral is $\dim c$-dimensional. If $c=\set x$ is a point of $Q$, then $\delta_{c}$ is the evaluation form  at $x$, which is denoted $\delta_x$.  We now slightly extend the family considered in Definition \ref{def:conditions-Z-discr} to allow stability conditions such that $\Img Z([V])=\delta_c(\udim_V)$ with $\dim c\neq n$.

  \begin{definition}\label{def:conditions-Z}
      Let $G\colon P\hookrightarrow Q$ be a grid function. We denote by $\calZ(G)$ the set of stability conditions over $\Persfp Q$ of the form
    \[
    [V]\mapsto Z([V])+ \sqrt{-1} \sum_{c\subset Q}\lambda_c\delta_c(\udim_V)
    \]
    where 
    \begin{itemize}
        \item  $Z$ lies in  $\calZ^\text{step}(G)$
        \item  the sum is  over  the faces of the bounded cubes  $(\cub G(x))_{x\in P}$ delimited by $Q$
        \item the $\lambda_c$ are nonnegative coefficients
    \end{itemize}
  \end{definition}
  \begin{rem}\label{rem:conditions-Z-as-delta-c}
If $Z=Z_{\alpha,\beta}$ lies in $ \calZ^\text{step}(G)$, we can write $\alpha$ as a finite  linear combination of $(\delta_{\cub { G}(x)})_{x\in  P}$ and $\beta$ is a countable strictly positive linear combination of $(\delta_{\cub {\widetilde G}(x)})_{x\in \widetilde P}$ for some extension $\widetilde G\colon \widetilde P\hookrightarrow Q$ of $G$.
  \end{rem}

   We use the same vocabulary and notations with $\calZ$ than with $\calZ^\text{eval}$. Namely, $\calZ(Q):=\bigcup_{G\text{ finite}} \calZ(G)$ and we say that $Z$ is $G$-\emph{discretisable} or that $G$ is \emph{adapted} to $Z$ when $Z\in\calZ(G)$.

\begin{rem}\label{rem:Z_G-Zn}
In the case where  $Q$ is a  cube of $\bbZ^n$, the  conditions of Definition  \ref{def:conditions-Z} can be simplified. Indeed, for each cube $c$ of $Q$, the function $\ind c$  is $G$-discretisable, so $\calZ^\text{step}(G)=\calZ(G)$. Moreover, $\id_Q$ is a refinement of any grid function over $Q$.  As a consequence,  a stability condition $Z$ over $Q$ lies in $\calZ(Q)$ if and only if it is of the form
 \[Z_{\alpha,\beta}([V])=\sum_{y\in Q}(\beta(y)+\sqrt{-1}\alpha(y))\dim V_y\]
 where $\alpha\colon Q\to \bbR$ is compactly supported and $\beta\colon Q\to \bbR_{>0}$ is integrable. Furthermore, any grid function $G\colon P\hookrightarrow Q$ where $Q$ is a cube of $\bbR^n$ or $\bbZ^n$ induces an assignment
 \[\left\{\begin{array}{ccc}
     \calZ(Q)&\to&\mathcal{Z}( P)\\
     Z&\mapsto & G^*Z
 \end{array}
\right.\] \end{rem}

\begin{prop}\label{prop:induction-step-fun-conserves-ss} Let $G\colon P\hookrightarrow Q$ be a finite grid function  and let $Z\in \calZ(G)$. Then for every nonzero $P$-persistence module $U$:
\[ U \text{ is }G^*Z\text{-semistable }  \ \Longrightarrow \  G_*U \text{ is }Z\text{-semistable}.\]
\end{prop}

We defer the proof of Proposition \ref{prop:induction-step-fun-conserves-ss} to the next Subsection. We now state our  existence theorem for HN filtrations of discretisable persistence modules:

\begin{theorem}\label{thm:existence_HN}Assume that either

\begin{enumerate}[label=(\roman*)]
    \item $Q$ is an upper semilattice and $Z\in\calZ^\text{eval}(Q)$
    \item  $Q$ is a cube of $\bbR^n$ or $\bbZ^n$ and $Z\in\calZ(Q)$
\end{enumerate}

Then, every nonzero discretisable  $Q$-persistence module $V\in\Persfp Q$ has a 
unique HN filtration along $Z$. Moreover, for each $f\colon P\hookrightarrow Q$ 
adapted to $Z$ with $P$ finite, the following diagram is well-defined and commutative:
\begin{equation}
    \label{eqn:dgm-HN-e}
 \begin{tikzcd}
        \Pers{P}\ar[rr,"\HN{ f^*Z}{}{-}"]  \ar[d," f_*"]&& \Fil{\Pers{  P}} \ar[d," f_*"]\\
        \Persfp{Q}\ar[rr,"\HN{Z}{}{-}"]  &&\Fil{ \Persfp{Q}}\\ 
    \end{tikzcd}
\end{equation}
where $f$ is either (i) an inclusion of upper semilattices  or (ii) a  grid function.
\end{theorem}

\begin{proof}
   Let $f\colon P\hookrightarrow Q$ be adapted to $Z$ with $P$  finite. Recall from Subsection \ref{subsec:HN-finite-poset} that   HN filtrations exist in $\Persfp{P}=\Pers {P}$ along any stability condition.  By combining Lemma \ref{lmm:preserve-ss-implies-HN} and Propositions \ref{prop:HN_existence_eval} and \ref{prop:induction-step-fun-conserves-ss}, we obtain that the diagram \eqref{eqn:dgm-HN-e} is well-defined and commutes.   
    
    Given  a nonzero $V\in \Persfp Q$,  there exists $f\colon P\hookrightarrow Q$ with $P$ finite adapted to both $Z$ and $V$.   By diagram \eqref{eqn:dgm-HN-e}, the HN filtration $\HN Z{}{V}\simeq \HN Z{}{f_*f^*V} $ is well-defined.
\end{proof}

The above Theorem can be generalised to the case of infinite grid functions: 
\begin{cor}\label{cor:commutes_infinite_grid}
    Let $G\colon P\hookrightarrow Q$ be a (potentially infinite) grid function and let $Z\in\calZ(G)$. Then, for $U\in \Persfp P$, we have $\HN Z{} {G_*U}=G_*\HN {G^*Z}{} U$.
\end{cor}
\begin{proof}
    By Remark \ref{rem:Z_G-Zn}, we have $G^*Z\in\calZ(P)$ and by Theorem \ref{thm:existence_HN}, both of the assignments $ \HN Z{} {-}\colon \Persfp P\to \Fil{\Persfp P }$ and $\HN {G^*Z}{} -\colon \Persfp Q\to\Fil{\Persfp Q}$ are well-defined. Let $U\in\Persfp P\setminus\set 0$, there is a finite grid function $G'\colon P'\hookrightarrow P$  such that $U$ is $G'$-discretisable. We can choose $G'$ such that  $Z$ lies in $\calZ_{G\circ G'}(Q)$. Indeed, we only need to refine $G'$ so that for all $p'\in P'$, either $\cub {G\circ G'}(p')$ does not intersect the support of $\Im Z$ or  $\cub {G'}(p')$ has cardinal 1.

    Assume that $U$ is $G^*Z$-semistable and let $U'$ be a $G'$-discretisation of $U$. By Lemma \ref{lmm:converse_ss_conserved}, $U'$ is $(G\circ G')^*Z$-semistable, and  by Proposition \ref{prop:induction-step-fun-conserves-ss}, the $Q$-persistence module $G_*U\simeq (G\circ G')_*U'$ is $Z$-semistable.   Finally, with  Lemma \ref{lmm:preserve-ss-implies-HN} applied to $G$, we have $\HN Z{} {G_*U}=G_*\HN {G^*Z}{} U$.
\end{proof}

\subsection{Proof of Proposition \ref{prop:induction-step-fun-conserves-ss}} Fix   a finite grid function $G\colon P\hookrightarrow Q$, a stability condition $Z\in\calZ(G)$ and a nonzero $P$-persistence module $U$ which is $G^*Z$-semistable.  Like in the proof of Proposition \ref{prop:HN_existence_eval}, the key to show that $G_*U$ is $Z$-semistable is to pixelate discretisable submodules $H_*X\subset G_*U$ while controlling their slope. In order to do so, we first study how the slope of  $H_*X$  evolves when the  grid function $H$ is perturbed. 

     Let  $H^\bullet\colon R\to Q$ be a family of  grid functions  indexed by   an interval $C$ of $\bbR$ of positive length. Namely, the (potentially improper) grid functions  $(H^x)_{x\in C}$ all have the same (co)domain, but their values are parametrised by $x$. Let $1\leqslant i\leqslant n$ and $p\in R_i$.
\begin{definition}\label{def:affine_at}
 The family $H^\bullet\colon R\to Q$ of grid functions is said to be \emph{affine at}  $(i,p)$ if $x \mapsto H^x_i(p)\in Q_i$ is affine  and for each other $1\leqslant i'\leqslant n $ and $p'\in R_{i'}$, the function $x\mapsto H^x_{i'}(p')$ is constant. 
\end{definition}

Henceforth, we assume that $H^\bullet$ is affine at $(i,p)$ with  $p\in R_i\setminus\set{\inf R_i,\sup R_i}$ and that the constant $\frac{\diff H^x_i(p)}{\diff x}$ is strictly positive.  The largest possible interval of parameters is the closed interval $C:=[x^-,x^+]$ such that  $H_i^{x^\pm}(p)$ coincides with the  constant $H^x_i(p\pm 1)$. The restriction of $H^x$ to 
 \begin{equation}\label{eqn:restr_remove_pt}
     R':=\bigg(\prod_{j=1}^{i-1}R_j\bigg)\times(R_i\setminus\set p)\times \bigg(\prod_{j=i+1}^nR_j\bigg)
 \end{equation}
 can be identified with a  grid function $\widehat H\colon \widehat R\to Q$. Indeed, the functions $x\mapsto H^x(p')$ are constant for each $p'\in R'$ and there is an isomorphism of posets $\widehat\tau$, independent of $x$, between $R'$ and a  cube $\widehat R$ of $\bbZ^n$.

 We denote respectively by $\pi_i$ and $e_i$ the $i$-th orthogonal projection and the $i$-th canonical basis vector in a cube of either $\bbR^n$ or $\bbZ^n$.

\begin{lemma}\label{lmm:affine_at_grid_fun}
  Using the above notation, let $H^\bullet\colon R\to Q$ be  affine at some $(i,p)$,  let $c$  be a face of a bounded cube  induced by an extension of  $\widehat H$ and let $X\in\Pers R$. 
    
    If  $\pi_i(c)$ is not a point, the function $a\colon x\in [x^-,x^+]\mapsto\delta_c(\udim_{H^x_*X})\in\bbR$  is affine. Otherwise, $a_{|(x^-,x^+]}$ is constant  and $a(x^-)- a(x^+)$ is a nonnegative linear combination of  $(\dim X_{r}-\dim X_{r-e_i})_{r\in\pi_i^{-1}(p)}$.
\end{lemma}

\begin{proof}
The function $a$ decomposes  as 
\[a(x)=\sum_{r\in R}\dim X_r  \prod_{j=1}^n  \delta_{\pi_j(c)}(\ind{\cub {H^x_j}(r_j)})\]
where $r_j$ is the $j$-th component of $r$. Observe that the function $x\mapsto\cub {H^x_j}(r_j)$  is constant unless $i=j$ and $r_j\in \set{p-1,p}$. Hence, we only  need to compute the  functions $b_{s}\colon x\mapsto \delta_{\pi_i(c)}(\ind{\cub {H^x_i}(s)}) $ for $s\in \set{p-1,p}$. 

By definition, $\pi_i(c)$ is either a face of a bounded interval delimited by $\widehat H_i$ or is included in an unbounded interval delimited by $\widehat H_i$. Hence, $b_s$ can only be nonzero for  $\pi_i(c)=[H^x_i(p-1),H^x_i(p+1)) $ or $\pi_i(c)=\set{H^x_i(p-1)}$. In the first case, we have $\cub{H^x_i}(s)\subset \pi_i(c)$ for all $x\in[x^-,x^+]$ so
\[b_{s}(x)= \len { \cub{H^x_i}(s)}=\begin{cases}
    H^x_i(p)-H^x_i(p-1)&\text{if } s=p-1\\
     H_i^x(p+1)-H_i^x(p)&\text{if } s=p
\end{cases}\]
which is affine. In the second case, we have 
\[b_s(x)= \ind{H^x_i(p-1)\in \cub{H_i^x}(s)}=\begin{cases}
    1-\ind{x^-}&\text{if } s=p-1\\
     \ind{x^-}&\text{if } s=p
\end{cases}\]
which is constant on $(x^-,x^+]$. We rewrite the function $a$ as 
\[a(x)=C+\sum_{r\in\pi_i^{-1}(p)} a_r(\dim X_r b_p+\dim X_{r-e_i}b_{p-1})\]
where $C=\displaystyle\sum_{\pi_i(r)\notin\set{p-1,p}}\dim X_r\delta_c(\ind{\cub{H^x}(r)})$  and $\displaystyle a_r= \prod_{j\neq i}\delta_{\pi_j(c)}(\ind{\cub {H^x_j}(r_j)}) $. The result follows from the expression of $b_s(x)$ and the fact that $C$ and $a_r$ are nonnegative constants. 
\end{proof}

\begin{prop*}[Restatement of Proposition \ref{prop:induction-step-fun-conserves-ss}]
    Given  a finite grid function $G\colon P\hookrightarrow Q$ and a stability condition $Z\in\calZ(G)$, the pushforward $G_*U$  of a  $G^*Z$-semistable  nonzero $P$-persistence module $U$ is $Z$-semistable.
\end{prop*}

\begin{proof}
     Let $U$ be a $G^*Z$-semistable nonzero $P$-persistence module and let $H\colon  R\hookrightarrow Q$ be a finite refinement of $G$. We prove by induction over the cardinal 
     \[N_H:=\left\lvert\bigsqcup_{i=1}^n \Img H_i\setminus \Img G_i\right\rvert\]
     that  every nonzero $H$-discretisable submodule $W\subset G_*U$ satisfies     $\mu_Z(W)\leqslant\mu_Z(G_*U)$. We denote by $X$ the $R$-persistence module $H^*W$. 
     
\noindent\textbf{\underline{Base Case}}: when $N_H=0$, we can assume that  $G=H$ and we have $\mu_Z(W)= \mu_Z(G_*X)=\mu_{G^*Z}(X)\leqslant \mu_{G^*Z}(U)=\mu_Z(G_*U)$. 

\noindent\textbf{\underline{Induction Case}}: assume that $N_H>0$  and that the result holds for all refinements $H'$ of $G$ such that $N_{H'}=N_H-1$. Let $1\leqslant i\leqslant n$ and let $p\in \Img H_i\setminus \Img G_i$. There is a family of finite grid functions $H^\bullet\colon R\to Q$ affine at $(i,p)$  such that $H=H^x$ for some $x$. Note that when  $p=\inf R_i$, one has $H(p)\in\cub G({-\infty})$ so $H^x_*X\simeq W$ for all $x$.  We will investigate later the case $p=\sup R_i$ and assume that $p\notin\set{\inf R_i,\sup R_i}$. Like in the above discussion, we parametrise $H^\bullet$ by $[x^-,x^+]$ such that $H_i^{x^\pm}(p)=H_i(p\pm1)$.

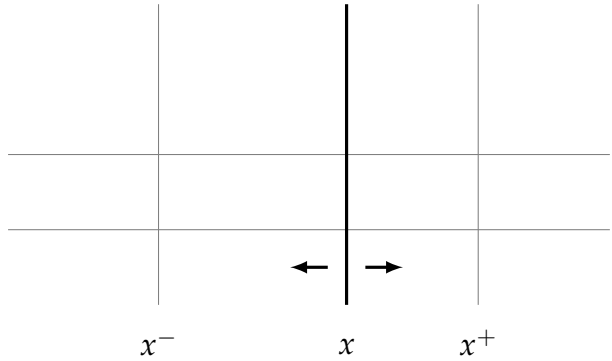
\begin{center}
\begin{minipage}{0.99\textwidth}\centering

\include{images/gridmove}

\end{minipage}
\end{center}

We consider the two $Q$-persistence modules $W^\pm:=(H^{x^\pm})_*X $ and we prove that the following partial function admits a maximum which is reached either at $x^-$ or at $x^+$
\[m\colon \left\{\begin{array}{ccc}
[x^-,x^+]&\longrightarrow& \bbR\\
    x&\longmapsto &\mu_Z(H^x_*X)
\end{array}\right..\]

\underline{$m$ is monotonic  on $(x^-,x^+)$}:  since $W\subset G_*U$ and $p\notin \Img G_i$, we have that for every $r\in\pi_i^{-1}(p)$, the linear map $X_{r-e_i}\to X_{r} $ is a restriction of the isomorphism $(G_*U)_{H(r-e_i)}\to (G_*U)_{H(r)}$ and is hence injective. By Lemma \ref{lmm:affine_at_grid_fun},  the nonnegative function $x\in[x^-,x^+]\mapsto  \beta(\udim_{H^x_*X})$ is affine. Moreover, since the family $(\dim X_r-\dim X_{r-e_i})_{r\in\pi_i^{-1}(p)}$ is nonnegative,   $x\in[x^-,x^+]\mapsto \alpha(\udim_{H^x_*X})$ is the sum of an affine function and of $\lambda \ind{x^-}$ with $\lambda\geqslant 0$. Hence, the function $m_{|(x^-,x^+)}$ is well-defined and    monotonic as the quotient of an  affine function by a strictly positive affine function. 

\underline{$m$ has a maximum and it is reached at $x^-$ or $x^+$}: since $R\setminus\Img (\Gfloor {\ }{H^{x^-}})=  \pi_i^{-1}(p-1)$ and $W\neq0$, the injectivity of $X_{r-e_i}\hookrightarrow X_{r} $ for $r\in\pi_i^{-1}(p)$ ensures that $W^-$  is never null. Thus, the function $m$ is well-defined and upper semicontinuous at $x^-$. If $W^+\neq0$, $m$ is also well-defined at $x^+$ and  monotonic on $(x^-,x^+]$. It thus admits a maximum which is reached at $x^-$ or $x^+$.  Otherwise, if $W^+=0$, we then have that  both $x\mapsto \alpha(\udim_{H^x_*X})$ and $x\mapsto \beta(\udim_{H^x_*X})$ are affine on $(x^-,x^+]$ and null at $x^+$.   Whence, $m$ is constant on $(x^-,x^+)$ and we have $m(x^-)\geqslant m(x) $ for all $x\in[x^-,x^+)$. 

\underline{Case $p=\sup R_i$}: if $p=\sup R_i$, we write $x^+:=+\infty$ and $W^+:=(H_{R'})_*X$. The function $x\in (x^-,x^+) \mapsto \alpha(\udim_{{H^x}_*X})$ is constant while  $x\in (x^-,x^+)\mapsto \beta(\udim_{{H^x}_*X})$ is decreasing and positive. Hence, like in the general case, $m$ is monotonic on $(x^-,x^+)$, upper semicontinuous at $x^-$ and -- if $W^+\neq 0$ -- continuous at $x^+=+\infty$. Moreover, as before, $W^+=0$ implies that  $m$ is constant on $(x^-,x^+)$. 

In any case, we found a nonzero $W'\in \set{W^-,W^+}$ such that $\mu_Z(W')\geqslant \mu_Z(W)$. 

 \underline{$W'$ is submodule of $G_*U$}: there is a finite grid function $\tau\colon P\hookrightarrow R$, independent of $x$,  such that $G=H^x\circ \tau$ for all $x\in[x^-,x^+]$. Indeed, using the notation in \eqref{eqn:restr_remove_pt}, each $H^x$ is a refinement of $\widehat H$ (with an associated finite grid function $\widehat \tau$ independent of $x$) which is itself a refinement of $G$. By definition, $W'=H^{x^\pm}_* H^*W$, and by Lemma \ref{lmm:abelianity-semilattice}(ii), 
\[W'\subset H^{x^\pm}_* H^*G_*U\simeq  H^{x^\pm}H^*H_*\tau_*U\simeq H^{x^\pm}_*\tau_*U \simeq G_*U. \]

Finally, $0 \neq W'\subset G_*U$ is $\widehat H$-discretisable and by induction hypothesis, $\mu_Z(W)\leqslant\mu_Z(W')\leqslant \mu_Z(G_*U)$. 
\end{proof}

\subsection{HN filtered rank invariants} For this subsection $Q$ is either $\bbR^n$ or $\bbZ^n$. Let $x\in Q$, the map of posets  
\[T_x\colon y \in Q\mapsto x+y \in Q\]
is called the $x$-\emph{shift}. Its pullback $T_x^*\colon V\in \Persfp Q\mapsto  V\circ T_x\in \Persfp Q$  induces an endofunctor of $\Persfp Q$ which we call the $x$-\emph{shift functor}.  

We fix  a  stability condition $Z\in\calZ(Q)\cup\calZ^\text{eval}(Q)$ like in Theorem \ref{thm:existence_HN}. Following the idea behind the skyscraper invariant of computing HN filtrations along translations of $Z$, we define the following invariant of discretisable $Q$-persistence modules:

\begin{definition}\label{def:rank filtrations}
    The  \emph{HN filtered rank invariant}  of $V\in\Persfp Q$ along $Z$ is the  filtration $(s_{Z,V}^\theta)_{\theta\in\bbR^\text{opp}}$ of integer-valued functions on $Q\times Q$  given by
     \[s_{Z,V}^\theta\colon \left\{\begin{array}{ccc}
        Q\times Q  & \to&\bbZ_+\cup\set{+\infty} \\
          (x,y)&\mapsto&\begin{cases}
              \rank (\HN{Z}{\theta}{T_x^*V}_{0\leqslant y-x})&\text{if }x\leqslant y\\
              +\infty&\text{otherwise.}  
          \end{cases}
     \end{array}\right.\]
\end{definition}

Theorem \ref{thm:existence_HN} ensures that $s^\theta_{Z,V}$ is   well-defined  and that for $(x,y)\in Q\times Q$, the integer $s_{Z,V}^\theta(x,y)$ can be computed in any suitable discretisation of $V$.  Note that by Lemma \ref{lmm:preserve-ss-implies-HN}, we can translate $Z$ instead of $V$. Namely, for $x\leqslant y$ in $Q$ and $V\in\Persfp Q$, we have
\[s^\theta_{Z,V}(x,y)= \rank(\HN{T^x_*Z}{}{V})_{x\leqslant y},\]
where $T^x_*Z:=(T_x^*)^*Z$ is the pullback of $Z$ by $T_x^*$ like in Equation \eqref{eqn:commute_Z}.

\begin{prop}
There is $\theta_\text{min}\in\bbR$ such that for  all $\theta\leqslant \theta_\text{min}\in\bbR$ and for all  $V\in\Persfp Q$, we have  \[s_{Z,V}^{\theta}=\rho_V.\]
\end{prop}

\begin{proof}
    Given $x\in Q$, by definition, $s_{Z,V}^\theta(x,\bullet)=\rho_V(x,\bullet)$ for $\theta$ small enough. We now find a uniform lower bound for $\theta$. Assume that  $Z=Z_{\alpha,\beta}$ lies in $\calZ(G)$ for some finite grid function $G\colon P\hookrightarrow Q$. Following Remark \ref{rem:conditions-Z-as-delta-c}, we write 
    \[\alpha=\sum_{c}a_c\delta_c\qquad \text{and}\qquad \beta=\sum_cb_c\delta_c\]
    where the sums are over the partition into cubes delimited by an extension of $G$. Given  $W\in \Persfp Q$, we have the inequalities
    \[\alpha(\udim_W)\geqslant \min(0,\min_c a_c) \sum_c \delta_c(\udim_W)\]\[\beta(\udim_W)\geqslant \min_c(b_c) \sum_c \delta_c(\udim_W)\]
    where the minima and sums are over  the bounded cubes $\cub G(x)$ delimited by $G$.  Let  $\theta_\text{min}$ be the real number $\min(0,\min_c a_c) / \min_c(b_c)$. If   $V^0\subsetneq V^1\subsetneq \dots\subsetneq V^{\ell-1}\subsetneq V^\ell=T_x^*V$ is the HN filtration of $T_x^*V$ along $Z$, we have   $\mu_Z(V^\ell/V^{\ell-1})\geqslant \theta_{\text{min}}$. Then for $\theta\leqslant \theta_{\text{min}}$, we have $\HN{Z}{\theta}{T_x^*V}=T_x^*V$ and $s_{Z,V}^\theta(x,y)=\rank (T_x^*V)_{0\leqslant y-x}=\rho_V(x,y)$.
\end{proof}

The \emph{continuous skyscraper} invariant is the HN filtered rank invariant along $Z_{\delta_{\set 0},\beta}$ where  $\beta\in\mathcal F(Q,\bbR)^*$ is any positive linear form. It is closely related to the  finite version of the skyscraper invariant defined in \eqref{eqn:skyscraper}. Indeed, by combining  Theorem \ref{thm:existence_HN} and \cite[Proposition 3.3]{HN_discr}, one gets that for all $x\in Q$, the function $ s^\bullet_{Z_{\delta_0,\beta},V}(x,\bullet)$ can be computed as the HN type of the discretisation $G^*V$ along the stability condition $G^*T^{x}_*Z_{\delta_{0},\beta}=(T_{-x}\circ G)^*Z_{\delta_{0},\beta}$ where $G$ is a finite grid function  adapted to $V$ whose image contains $x$. Moreover, we will see in Appendix \ref{append:semialgebraic}, that if $\beta$ is induced by a step function, then $ s^\bullet_{Z_{\delta_0,\beta},V}(x,\bullet)$ only needs to be evaluated at finitely many values of $x\in Q$.  Conversely, given a finite grid function $G\colon P\hookrightarrow Q$, the HN type of $U\in\Pers P$ along a skyscraper weight at $p\in P$ can be read from the filtered rank invariant of $G_*U$ along the stability condition $Z_{\delta_{G(p)},\beta}$   for some choice of positive $\beta\in\mathcal F(Q,\bbR)^*$.

\section{Stability theorems}\label{sec:stability}

 The  rank invariant, equipped with the erosion distance introduced by Patel \cite{patel2018generalized}  has been shown to be  stable \cite{kim2021spatiotemporal} with respect to the interleaving distance. We extend this result to the HN filtered rank invariants obtained in Definition \ref{def:rank filtrations}. The HN filtrations also induce filtered persistence landscapes which  are a generalisation of the multiparameter persistence landscapes introduced by Vipond \cite{vipond2020multiparameter}. As a consequence of the above result, this last invariant is also stable. For this section, the poset $Q$ will be either $\bbR^n$ or $\bbZ^n$. 

\subsection{Functoriality of HN filtrations}

The key to adapt the existing stability theorems to the case of HN filtered rank invariants   will be the functoriality of HN filtrations. We restate {\cite[Theorem 2.8]{hille2002stable}} using Remark \ref{rem:quiver_poset}:

\begin{prop}\label{prop:functorial_HN-finite}
 Let $P$ be a finite poset  and let $Z$ be a stability condition on $\Pers P$. Then for every map  $f\colon U\to U'$  in $\Pers P$ and every  $\theta\in\bbR$
\[f(\HN Z\theta U)\subset\HN Z\theta {U'}. \]
\end{prop}

\begin{cor}\label{cor:functorial-HN-continuous}
    Let $f\colon V\to V'$ be a map in $\Persfp Q$ and let $Z\in\calZ(Q)$. Then for every  $\theta\in\bbR$
\[f(\HN {Z}\theta V)\subset\HN {Z}\theta {V'}. \]
\end{cor}
\begin{proof}
Assume  that $V$ and ${V'}$ are in $\Persfp Q$ and choose  a finite grid function $G\colon P\hookrightarrow Q$ such that $Z\in\calZ(G)$ and   $V$ and ${V'}$ both have a $G$-discretisation, respectively $U$ and $U'$. By  Proposition \ref{prop:functorial_HN-finite}, $G^*f$ induces a  map of $ P$-persistence modules $\widetilde f\colon\HN {G^*Z}\theta {U}\to\HN{G^*Z}\theta {U'}$ where $\widetilde f=G^*( G_*U\simeq V\overset f\to V'\simeq G_*U')$. By  Theorem \ref{thm:existence_HN}, $\widetilde f$ extends to a map
\[G_*\widetilde f\colon \HN{Z}\theta {G_*U}\to \HN{Z}{\theta}{G_*U'}.\]
Finally, by exactness of $G_*$, the map $\HN{Z}\theta V\simeq \HN{Z}\theta {G_*U}\overset{G_*\widetilde f}\to \HN{Z}\theta {G_*U'}\simeq \HN{Z}\theta {V'}$ is the restriction of $f$.  Hence, $f(\HN{Z}\theta V)\subset \HN{Z}{\theta}{V'}$ in $\Persfp Q$. 
\end{proof}

\subsection{Stability of HN filtered rank invariants}
 The $x$-\emph{shift map} of $V$ is the map of $Q$-persistence modules  $\shift xV\colon  V\to  T_x^*V$  defined by $(\shift{x}{V})_y:=V_{y\leqslant x+y}$.

For $\varepsilon>0$,  $\Vec \varepsilon$ denotes the vector $(\varepsilon,\dots,\varepsilon)\in Q$ and we see $Q$ and $\bbZ_+\cup\set{+\infty}$ as posetal categories (associated with the standard partial order). We recall the definition of the following two extended pseudometrics:

\begin{definition}[{\cite{lesnick2015theory}}]\label{def:interleaving-distance}Given $\varepsilon>0$, an $\varepsilon$-\emph{interleaving} between two $Q$-persistence modules $V$ and $W$ is a pair of maps  $f\colon V\to T_{\vec\varepsilon}^*W$ and $g\colon W\to T_{\vec\varepsilon}^*V$ in $\Pers Q$ such that
    \[T_{\vec\varepsilon}^*g\circ f= \shift{2\vec\varepsilon}V\hspace{3em}\text{and}\hspace{3em} T_{\vec\varepsilon}^*f\circ g = \shift{2\vec\varepsilon}W.\]
    The \emph{interleaving distance} between $V$ and $W$, denoted by $d_I(V,W)$, is the infimum over $\varepsilon\geqslant0$ such that there is an $\varepsilon$-interleaving between $V$ and $W$.
\end{definition}

\begin{definition}[\cite{patel2018generalized,kim2021spatiotemporal}]
    Consider two  functors \[F,G\colon Q^\text{opp}\times Q\to (\bbZ_+\cup\set{+\infty})^\text{opp}.\]   For $\varepsilon>0$, we say that there is an $\varepsilon$-\emph{erosion} between $F$ and $G$ if for all $(a,b)\in Q^\text{opp}\times Q$ 
    \[F(a-\Vec \varepsilon,b+\Vec \varepsilon)\leqslant G(a,b)\qquad\text{and}\qquad G(a-\Vec \varepsilon,b+\Vec \varepsilon)\leqslant F(a,b).\]
    The \emph{erosion distance} between $F$ and $G$, denoted by $d_E(F,G)$, is the infimum over $\varepsilon\geqslant0$ such that there is an $\varepsilon$-erosion between $F$ and $G$.
\end{definition}

An example of such functors is the rank invariant $\rho_V$ of some persistence module $V$. In particular, it is known to  be stable with respect to the erosion and interleaving distances:

\begin{theorem}[{\cite[Theorem 6.2]{kim2021spatiotemporal}}{\cite[Theorem 3.11]{puuska2020erosion}}]\label{thm:stab-rk-inv}
   Given two $Q$ persistence modules $V$ and $W$,  we have $d_E(\rho_V,\rho_W)\leqslant d_I(V,W)$    where $d_I$ is the interleaving distance.
\end{theorem}

We fix a stability condition $Z\in\calZ(Q)$ and two  persistence modules $V$ and $W$ in $\Persfp{Q}$. For convenience, we drop $Z$ from the $s^\theta_{Z,V}$ notation. The main argument of the proof of Theorem \ref{thm:stab-rk-inv} can be adapted to our setting and results in the following lemma:

\begin{lemma}\label{lmm:stab_main}
Let $x\leqslant x'\leqslant y'\leqslant y$ in $Q$ and assume the existence of maps  
\[f\colon V\to T_{x'-x}^*W\text{\quad and \quad}g\colon W\to T_{y-y'}^*V\]
such that $g\circ f=\shift{(y-x)-(y'-x')}V$. Then for all $\theta\in\bbR$, one has 
$s^\theta_{V}(x,y)\leqslant s^\theta_{W}(x',y')$.
\end{lemma}
\begin{proof}
    The properties of  $f$ and $g$ make the following  diagram well-defined and commutative
\[\scalebox{0.9}{\begin{tikzcd}[ampersand replacement = \&]
\HN{Z}\theta {T_x^*V}_{0}\ar[dr,"f_x"]\ar[rrr]\&\&\& \HN{Z}\theta {T_x^*V}_{y-x}\\
\& T_x^*f(\HN{Z}\theta {T_x^*V})_{0}\ar[r]\&T_x^*f(\HN{Z}\theta {T_x^*V})_{y'-x'}\ar[ru,"g_{y'}"]
\end{tikzcd}}
\]
Here, the top horizontal map is the restriction of $V_{x \leqslant y}=(T_x^*V)_{0\leqslant y-x}$ to the submodule $\HN{Z}\theta {T_x^*V}\subset T_x^*V$. By definition, its rank is $s_{V}^\theta(x,y)$.  The bottom horizontal map is the restriction of $W_{x'\leqslant y'}=T_{x'}^*W_{0\leqslant y'-x'}$ to the submodule $(T_x^*f)(\HN{Z}\theta {T_x^*V})\subset T_{x'}^*W$. We denote its rank as $s(x',y')$. 

Since the above diagram commutes, we have $s_{V}^\theta(x,y)\leqslant s(x',y')$. Moreover, Corollary \ref{cor:functorial-HN-continuous} implies that $(T_x^*f)(\HN{Z}\theta {T_x^*V})\subset \HN{Z}\theta {T_{x'}^*W}$. Finally, $s^\theta_{V}(x,y)\leqslant s(x',y')\leqslant s_W^\theta(x',y')$.  
\end{proof}

\begin{theorem}\label{thm:stability_erosion} Let $Z\in\calZ(Q)$ be a stability condition over $Q$ and let $V,W\in\Persfp{Q}$ be two $Q$-persistence modules, then for each  $\theta\in\bbR$, their HN filtered rank invariant $s_{Z,V}^\theta$ and $s_{Z,W}^\theta$ are functors    $Q^\text{opp}\times Q\to(\bbZ_+\cup \set{+\infty})^\text{opp}$. Moreover,
\[\sup_{\theta \in\bbR}d_E(s_{Z,V}^\theta,s_{Z,W}^\theta)\leqslant d_I(V,W).\]
\end{theorem}

\begin{proof}Let $V\in\Persfp{Q}$ and let $x\leqslant x'\leqslant y'\leqslant y$, the shift maps $\shift{x'-x}V$ and $\shift{y-y'}V$ satisfy the conditions of Lemma \ref{lmm:stab_main}. Hence,  $s^\theta_{V}(x,y)\leqslant s_{V}^\theta(x',y')$ or in other words $s^\theta_{V}:Q^\text{opp}\times Q\to (\bbZ_+\cup\set{+\infty})^\text{opp}$ is a functor.

Let $V,W\in\Persfp{Q}$ and assume that there is an $\varepsilon$-interleaving $f\colon V\leftrightarrow W:g$. Let $\theta \in \bbR$, we show that $d_E(s_{V}^\theta,s_W^\theta)\leqslant\varepsilon$. For $a\leqslant b$ in $Q$,  the maps $f$ and $g$ satisfy the conditions of Lemma \ref{lmm:stab_main}  with $(x\leqslant x'\leqslant y'\leqslant y):=(a-\Vec \varepsilon\leqslant a\leqslant b\leqslant b+ \Vec \varepsilon)$.  Hence, $s_{V}^\theta(a-\Vec\varepsilon,b+\Vec\varepsilon)\leqslant s^\theta_W(a,b)$ and by symmetry $s_W^\theta(a-\Vec\varepsilon,b+\Vec\varepsilon)\leqslant s^\theta_V(a,b)$.
\end{proof}

\begin{rem}\label{rmk:null-Z-stab}
   Theorem \ref{thm:stab-rk-inv} for discretisable persistence modules is a special case of Theorem \ref{thm:stability_erosion} obtained when the image of $Z\colon K(\Persfp Q)\to \bbC$ is included in $\bbR$.   
\end{rem}

\begin{nota}\label{nota:HN-dist}Given two $Q$-persistence modules $V,V'\in\Persfp Q$ and a stability condition $Z\in\calZ(Q)$, we write
   \[d_{\bf HN}(V,V'):=\sup_{\theta \in\bbR}d_E(s_{Z,V}^\theta,s_{Z,W}^\theta)\]
   and we call this extended pseudometric the \emph{HN distance} along $Z$.
\end{nota}

\subsection{Stability of HN filtered landscapes} The HN filtered rank invariants induce filtered persistence landscapes. We show that those are also stable.

\begin{definition}Given a functor $s\colon Q^\text{opp}\times Q\to (\bbZ_+\cup \set{+\infty})^\text{opp}$, the \emph{persistence landscape} of $s$ is the function $\lambda_s\in \mathrm L^\infty(\bbN\times Q)$ given for $(k,x)\in\bbN\times Q$ by
\[\lambda_s(k,x):=\sup\set{\varepsilon>0\mid s(x-h,x+h)\geqslant k,\forall \|h\|_\infty\leqslant \varepsilon}.\]  
\end{definition}

The persistence landscape introduced in \cite{persistencelandscapes,vipond2020multiparameter} is the invariant  $V\in \Pers Q\mapsto \lambda_{\rho_V}\in \mathrm L^\infty(\mathbb N\times Q)$. They are stable with respect to the interleaving distance \cite[Theorem 30]{vipond2020multiparameter}. 

\begin{prop}\label{prop:landscape_erosion}
    Given two functors $r,s\colon Q^\text{opp}\times Q\to (\bbZ_+\cup \set{+\infty})^\text{opp}$ we have
    \[\lVert \lambda_r-\lambda_s \rVert_\infty \leqslant d_E(r,s).\]
\end{prop}
\begin{proof}
   Let $\varepsilon>d_E(r,s)$ and let $k,x\in\bbN\times Q$. We  show that $|\lambda_r(k,x)-\lambda_s(k,x)|\leqslant\varepsilon$. Without loss of generality,  we restrict ourselves to the case where $\varepsilon\leqslant \lambda_r(k,x)\geqslant\lambda_s(k,x)$. Let $h\in Q$ such that $\lVert h\rVert_\infty\leqslant \lambda_r(k,x)-\varepsilon$. By definition,
\[s(x-h,x+h)\geqslant r(x-h-\Vec\varepsilon,x+h+\Vec\varepsilon)\geqslant k.\]
Hence, $\lambda_s(k,x)\geqslant \lambda_r(k,x)-\varepsilon$ and $\lambda_r(k,x)- \lambda_s(k,x)\leqslant d_E(r,s)$.
\end{proof}

As before, let $Z\in\calZ(Q)$  and let $V\in \Persfp{Q}$. The \emph{HN filtered persistence landscape} of $V$ along $Z$ is the filtration $(\lambda_{s_{Z,V}^\theta})_{\theta\in\bbR^\text{opp}}$ of $\lambda_{\rho_V}$ in $\mathrm L^\infty (\bbN\times  Q)$. The stability of this last invariant is a direct corollary of Theorem \ref{thm:stability_erosion} and Proposition \ref{prop:landscape_erosion}.

\begin{cor}\label{cor:landscape-stability} Given   two discretisable persistence modules $V$ and $W$, their HN filtered persistence landscapes satisfy
\[\sup_{\theta\in\bbR}\lVert\lambda_{s_V^\theta}-\lambda_{s_W^\theta}\rVert_\infty\leqslant d_I(V,W).\]
\end{cor}

\appendix

\section{Discreteness of HN filtered rank invariants}\label{append:semialgebraic}

Let $Q$ be either $\bbZ^n$ or $\bbR^n$, let $V$ be a discretisable $Q$-persistence module and let $Z\in\calZ(Q)$ be a discretisable stability condition over $Q$. Example \ref{ex:compute-s} shows that, in general,  the HN filtered rank invariants over $Q$ are not   discretisable functions in the sense of Subsection \ref{subsec:HN_abelian}. We show instead that when $\Im Z$ is nonnegative, the HN filtered rank invariants over $Q$ are constant on the parts of a subdivision of its domain into a finite  number of semialgebraic sets.

\begin{ex}\label{ex:compute-s}
    Let $Q=\bbR^n$, $V=\K_{[0,1)^n}$,  $\alpha=\delta_0$ and $\beta\colon f\mapsto \int_{\bbR^n}bf$ where $b>0$  is an  integrable step function such that $b_{|[0,1)^n}=1$. Given $\theta>0$ and $x\leqslant y$ in $\bbR^n$, the rank $s_{Z_{\alpha,\beta},V}^\theta(x,y)$ lies in $\set{0,1}$. Moreover, $s_{Z_{\alpha,\beta},V}^\theta(x,y)=1$  if and only if $x\leqslant y$ lie in $[0,1)^n$ and $\mathrm{vol\ } (\subrep x\cap[0,1)^n)\leqslant 1/\theta$.  This last condition is a polynomial inequality of degree  $n+1$ in $(x,y,\theta)\in\bbR^{n}\times\bbR^n\times\bbR$. 
\end{ex}

The wall-chamber structure of the space of stability conditions has been well-studied for representations of finite quivers \cite{hille2002stable} and more generally for modules over a finite-dimensional algebra \cite{brustle2019wall}. For the purpose of this paper, we use the following definitions inspired by \cite[Section 3]{hille2002stable}.  Let $P$ be a finite poset and let   $U\in \Pers P$. Following equation \eqref{eqn:stab_cond_quiver}, we see a stability condition $\widehat Z$ over $P$ as a vector of $\bbR^P\times \bbR_{>0}^P$. A \emph{wall system} $\set{W^i}_I$ for $U$  is a set of codimension 1  algebraic subvarieties of $\bbR^{P}\times \bbR^P$  such that for each $I'\subset I$, and each   connected component $C$ of    $ (\bigcap_{I'} W^i)\setminus (\bigcup_{I\setminus I'} W^i)$, the HN filtration $0\subsetneq U^1\subsetneq \dots\subsetneq U^\ell =U$ of $U$ along $\widehat Z$ does not depend on $\widehat Z\in C\cap (\bbR^P\times\bbR_{>0}^P)$.  In that case, the semialgebraic sets  $C\cap (\bbR^P\times\bbR_{>0}^P)$ are called \emph{chambers} and each $W^i$ is called a \emph{wall}.  A key observation is that  the codimension 1 algebraic subvarieties among
  \begin{equation}
      \label{eqn:walls-finite-quiver}\left(\set{ \widehat Z\in \bbR^P\times \bbR_{>0}^P\ \big|\ \frac{\Im\widehat Z(d)}{\Re\widehat Z(d)}= \frac{\Im\widehat Z(d')}{\Re\widehat Z(d')} }\right)_{\ 0\leqslant d,d' \leqslant \udim_U}
  \end{equation}
  provide a finite wall system for $U$. 

   In order to extend the wall-chamber structure to the infinite case, we first  study common refinements between two families of grid functions. Let  $(G^{1,x}\colon P^1\hookrightarrow Q )_x$ and $(G^{2,x}\colon P^2\hookrightarrow Q )_x$ be two families of  grid functions over $Q$ whose values (but not their domains) depend on a  parameter $x\in C$. An \emph{independent  common refinement} of $(G^{1,x})_x $ and $(G^{2,x})_x$ is a choice for each $x$ of a  common refinement $G^x\colon P\hookrightarrow Q$ of $G^{1,x}$ and $G^{2,x}$ such  that there are two grid functions $\tau^1\colon P^1\hookrightarrow P$ and $\tau^2\colon P^2\hookrightarrow P$ independent of $x$ which fit for all $x$  into the following commutative diagram:
\begin{equation}
    \label{eqn:independent-ref}\begin{tikzcd}
    P^1\ar[dr,swap,hookrightarrow,"G^{1,x}"]\ar[r,hookrightarrow,"\tau^1"]&P\ar[d,hook,"G^x"]\ar[r,hookleftarrow,"\tau^2"]&P^2\\
    & Q \ar[ur,swap,hookleftarrow,"G^{2,x}"]&
\end{tikzcd}
\end{equation}

\begin{lemma}\label{lmm:refin}
    If for  $1\leqslant i\leqslant n$ and $(p^1,p^2)\in P_i^1\times P_i^2$, the  functions
    \begin{equation}
        \label{eqn:supp-lemma-refin}
        x\mapsto \left(G_i^{1,x}(p^1)- G_i^{2,x}(p^2)\right)
    \end{equation}
  have all constant sign (in $\set{-1,0,1}$), then $(G^{1,x})_x $ and $(G^{2,x})_x$ have an independent   common refinement $(G^x)$ whose image is $\Img G^x=\prod_i\Img G_i^{1,x}\cup\Img G_i^{2,x}$.
    \end{lemma}
\begin{proof}We first assume that $n=1$.  A minimal common refinement $G\colon P\hookrightarrow Q$ of two 1-dimensional grid functions $G^1$ and $G^2$ over $Q$ is a grid function  $G\colon P\cong \Img G^1\cup\Img G^2\subset Q$ and is unique up to translations of $P\subset \bbZ$. For $k\in\set{1,2}$, the grid function $\tau^k:=(G_{|P\to \Img G})^{-1}\circ G^k$ from $P^k$ to $P$ is entirely determined by the signs of $\tau^k(p^k)- p$ for  $p^k\in P^k$ and $p\in P$.  By minimality, $G(p)$ is of the form $G^\ell(p'^\ell)$ with $p'^\ell\in P^\ell$ and $\ell\in \set{1,2}$ and we have
    \[ \sign(\tau^k(p^k)- p)=\sign( G^k(p^k)- G(p)) =\begin{cases}
        \sign( p^k-p'^k)&\text{ if }\ell=k\\
        \sign(G^k(p^k)- G^\ell(p'^\ell))&\text{ otherwise}
    \end{cases}\]
    Hence, if we parametrise $G^1$ and $G^2$ by $x$, the grid function $\tau^k$  clearly fits into \eqref{eqn:independent-ref} and only depends on $P^k$ and the signs of the functions in \eqref{eqn:supp-lemma-refin}.

    Assume now that $n>1$.  For  each $1\leqslant i\leqslant n$, the existence of the $i$-th component of $G^x$ is obtained by applying the result of the Lemma to  $(G^{1,x}_i)_x $ and $(G^{2,x}_i)_x$. 
\end{proof}

  A family of grid functions $G^x\colon P\hookrightarrow Q$ indexed by $x$ is said to be \emph{affine} if $x\in C\mapsto G^x(p)\in Q$ is affine in $x$ for each $p\in P$. For $x\in Q$, let $T_x\colon Q\to Q$ denote the $x$-shift and let $G^0\colon \set 0\to Q$ denote the null grid function over $Q$.  Let $G^V\colon P^V\hookrightarrow Q$ and $G^Z\colon P^Z\hookrightarrow Q$ be finite grid functions adapted to respectively $V$ and $Z$. We further choose an extension $\widetilde G^Z\colon \bbZ^n\hookrightarrow Q$ of $G^Z$ such that $\Re Z$ is induced by a $\widetilde G^Z$-discretisable function.

\begin{lemma}\label{lmm:s-poly-bounded-C}  There exists a partition $\mathcal P$ of $Q$ into a finite number of cubes  such that for each part $C\in\mathcal P$, there is a finite affine independent common refinement  $(G^x\colon P\to Q)_{x\in C}$    of $(T_{-x}\circ G^V)_x$, $(G^Z)_x$ and $(G^0)_x$ so that the function 
\[x\in C\mapsto (G^x)^*Z\in\bbR^P\times\bbR_{>0}^P\]
is a polynomial in $x$ whenever $C$ is bounded.
\end{lemma}

\begin{proof} 
 \ 

\underline{Construction of $\mathcal P$ and $G^x$}:  The functions in \eqref{eqn:supp-lemma-refin} with $((G^{1,x}),(G^{2,x}))$ being every possible pair of elements of $\set{(T_{-x}\circ G^V),G^Z,G^0}$ delimit a finite partition $\widehat {\mathcal P}$ of $Q$ where inside each part, the sign (in $\set{-1,0,1}$) of each of those functions is constant. Moreover, since $G^Z$ and $G^0$ are  independent of $x$ and each $(T_{-x}\circ G^V)_i$ is affine in  $x_i$ and independent of $(x_j)_{j \neq i}$, the parts of this partition are cubes. 

Let  $C$ be a  cube in  $\widehat {\mathcal P}$,  by Lemma \ref{lmm:refin}, there is a   finite independent common refinement $  G^x\colon   P\hookrightarrow Q$ of $(T_{-x}\circ G^V)$, $G^Z$ and $G^0$ indexed by $x\in C$. And we can choose $(  G^x)_x$ to be affine since $(T_{-x}\circ G^V)$, $G^Z$ and $G^0$  are all affine. 

Let $C$ be a  bounded cube in  $\widehat {\mathcal P}$. Since $(  G^x)$ is affine, there is a bounded cube $B$ of $Q$ which contains $\Img   G^x$ for each $x\in C$. Hence, $\Img \widetilde G^Z_i\cap B_i$ is finite for all $i$, and  by Lemma \ref{lmm:refin},  we can further partition $C$ into  a finite number of cubes $C'$ inside which $(  G^x)$  and $\widetilde G^Z$ have an affine independent common refinement $(\widetilde G^x\colon\bbZ^n\to Q)$. Let $\mathcal P$ be the refinement of $ \widehat {\mathcal P}$ induced by this partition of each bounded  $C\in \widehat {\mathcal P}$.

\underline{Semialgebraicity of $(G^x)^* Z$}: Assume that $C\in\mathcal P$ is bounded. For $x\in C$, we denote by $Z^x=(\Im Z^x,\Re Z^x)\in \bbR^P\times\bbR_{>0}^P$ the stability condition $(G^x)^*Z$. Let $\tau\colon P\to\bbZ^n$ be a grid function such that $\widetilde G^x\circ \tau = G^x$ for all $x\in C$. Fix $p\in P$, one can check that  for $z\in \bbZ^n$,  
\[\cub{\widetilde G^x}(z)\cap \cub{G^x}(p) =\begin{cases}
   \cub{\widetilde G^x}(z)&\text{if } z\in \cub{\tau}(p)\\
   \emptyset &\text{otherwise}
\end{cases} \]

So for a face $c$ of a cube of the form $\cub{\widetilde G^x}(z)$, the map $x\in C \mapsto \delta_{c}(\ind{\cub{G^x}(p)})$ is polynomial of degree at most $n$.    Since $\Im Z^x_p$ is a linear combination of finitely many  such maps, it is also polynomial in $x\in C $. By independence of $\widetilde G^x$, there is $b\colon \bbZ^n\to \bbR_{>0}$ such that for every $x\in C $ we have $\Re Z^x([V])=\int_Q (\widetilde G^x_*b)\udim_V$. The function $x\in C \mapsto \Re Z^x_p$ can now be expressed as the pointwise limit  when $N\to +\infty$ of the polynomials
\[R_N\colon x\in C \mapsto\sum_{z\in [-N,N]^n\cap\cub \tau(p)} b_z\vol({\cub{\widetilde G^x}(z)}).\]
whose variables are the $x_i$ such that $C_i$ is not a point. Finally, since the degree of $R_N$ is bounded by $n$, by Lagrangian approximation, the coefficients of $R_N$ are also bounded and the limit $\Re Z^x_p$ is itself polynomial.  
\end{proof}

\begin{definition}[{\cite[Section 1]{mccrory1997algebraically}}]\label{def:constructible}
    A function $\phi\colon X\to \bbZ$ from a real algebraic set $X$ is called \emph{semialgebraically constructible} if it can be written as a finite sum
    \[\phi=\sum_i m_i\ind{X_i}\]
    where for each $i$,  $m_i\in\bbZ$ and $X_i$ is a semialgebraic subset of $X$.
\end{definition}

We denote by $Q_+^2$ the half-space $  \set{(x,y)\in Q^2\mid x\leqslant y}$ of $Q^2$.

\begin{prop}\label{prop:semialgebraic}  
    Assume that $\Im Z$ is nonnegative. The HN filtered rank invariant 
    \[s_{Z,V}^\bullet\left\{\begin{array}{ccc}
   Q_+^2   \times \bbR& \longrightarrow&\bbZ\\
        (x,y,\theta)&\longmapsto &s_{Z,V}^\theta(x,y)
    \end{array}\right.\]
    is a semialgebraically constructible   function.
\end{prop}

\begin{proof}Let $C=C_1\times\dots\times C_n$ be a cube of a partition given by Lemma \ref{lmm:s-poly-bounded-C} and let $\widetilde U\in \Pers {P^V}$ be a $G^V$-discretisation of $V$.

     \underline{Unbounded below}: If one of the components $C_i$ is unbounded below, then for all $x\in C$, we have 
      \[s_{Z,V}^\theta(x,y)=\rank\HN Z\theta {T_{x}^*V}_{0\leqslant y-x}=\begin{cases}
          \rank V_{x\leqslant y} &\text{if }\theta\leqslant 0\\
          0&\text{otherwise}.
      \end{cases}\]Indeed, for all $(p^V,p^Z)\in P_i^V\times P^Z_i$, the function $x_i\in C_i \mapsto G_i^V(p^V)-x_i-G^Z_i(p^Z)\in Q_i$ is of constant sign and is  hence positive -- meaning that the supports of $T^*_xV$ and $\Im Z$ are disjoint.   

  \underline{Bounded}:  If $C$ is bounded, then  using the result and notations of  Lemma \ref{lmm:s-poly-bounded-C}, there is $U\in \Pers P$ such that  $T^*_xV\simeq G^x_*U$ for all $x\in C$. Indeed, one can choose $U= \tau_*\widetilde U$ where $G^x\circ \tau =G^V$. By Theorem \ref{thm:existence_HN}, we obtain $\HN{Z}{\theta}{G^x_*U} = G^x_*\HN{(G^x)^*Z}{\theta}{U}$ for every $\theta\in \bbR$, and in particular,
     \[s_{Z,V}^\theta(x,y)=\rank\HN {(G^x)^*Z}\theta {U}_{\lfloor 0 \rfloor_{G^x}\leqslant\lfloor y-x \rfloor_{G^x}} .\]

    Since $x\in C\mapsto (G^x)^*Z\in \bbR^P\times\bbR^P_{>0}$ is polynomial, the inverse image of the chambers defined in \eqref{eqn:walls-finite-quiver} are again semialgebraic.   Fix one such semialgebraic set $C'$, the HN filtration  $0\subsetneq U^1\subsetneq \dots\subsetneq U^\ell=U$ of $U$ along $(G^x)^*Z$ does not depend on $x\in C'$. Let $p\in P$ and $0\leqslant i\leqslant \ell$, the set $S_{C',p,i}$ of all $(x,y,\theta)\in C'\times Q\times \bbR$ such that 
        \begin{equation}\label{eqn:semialg-slope}
        y-x\in \cub{G^x}(p) \hspace{3em}\text{and}\hspace{3em} \mu_{(G^x)^*Z}(U^i/U^{i-1})\geqslant \theta> \mu_{(G^x)^*Z}(U^{i+1}/U^i)
    \end{equation}
      is semialgebraic. There are finitely many $S_{C',p,i}$   and inside each of them,  $s^\theta_{Z,V}(x,y)$ is the constant given by $\rank (U^i)_{0\leqslant p}$.

   \underline{Unbounded above}:   Finally, assume that each component of  $C=C_1\times\dots\times C_n$ is either bounded or unbounded above. Let $I_b$ and $ I_u$ denote  the set of indices $1\leqslant i \leqslant n$ where $C_i$ is respectively bounded and unbounded above. Let $\widetilde x$ be the join of $0$ and of $G^Z(\inf P^Z)$; for $W\in \Persfp Q$, define the discretisable submodule $W^{\geqslant \widetilde x} $  of $W$ determined by the subspaces
      \[W^{\geqslant \widetilde x}_x:=\begin{cases}
          W_x&\text{if } x\geqslant \widetilde x\\
          0& \text{otherwise}.
      \end{cases} \]
 Since $\Im Z$ is nonnegative, if $0\subsetneq W^{\geqslant \widetilde x} \subsetneq W$, we have $\mu_Z(W^{\geqslant \widetilde x} )>\mu_Z(W)$. Hence, every $Z$-semistable $W$ satisfies $W=W^{\geqslant\widetilde x}$ and 
 \[{\HN{Z}{}{V}}^{\geqslant\widetilde x}=\HN{Z}{}{V^{\geqslant\widetilde x}}.\]
 Let $x_u\in \prod_{i\in I_u} C_i$, the cube $C':=\set {x_u}\times\prod_{i\in I_b}C_i$ is bounded and by the previous case $C'\times Q\times \bbR$ can be partitioned into finitely many semialgebraic  sets $(S_j)_j$ such that each $(s^\bullet_{Z,V})_{|S_j} $ is constant. Given $(x,y)\in C\times Q$, we denote $x':=x_u+\pi_{I_b}(x)$ and $y':= y + x'-x$ where $\pi_{I_b}$ is the projection onto $\prod_{I_b}C_i$.   Moreover, for $i\in I_u$ and $p_i^V\in P^V_i$, the expression $x_i\in C_i \mapsto G^V_i(p^V_i)-x_i-\widetilde x_i$ has constant sign and is hence negative. Thus, we have $(T_x^*V)^{\geqslant\widetilde x}\simeq (T_{x'}^*V)^{\geqslant\widetilde x}$ and  for every  $y\geqslant x$,
 \begin{align*}
     s^\bullet_{Z,V}(x,y)=\rank [(\HN{Z}{\bullet}{T_x^*V})^{\geqslant\widetilde x}]_{0\leqslant y-x}&=\rank \HN{Z}{\bullet}{(T_x^*V)^{\geqslant\widetilde x}}_{0\leqslant y-x}\\
     &=\rank \HN{Z}{\bullet}{(T_{x'}^*V)^{\geqslant\widetilde x}}_{0\leqslant y'-x'}\\
     &=s^\bullet_{Z,V}(x',y').
 \end{align*}
Since the map $\pi\colon (x,y,\theta)\mapsto (x',y',\theta)$ is affine, the set $C\times Q\times\bbR$ can be partitioned into finitely many semialgebraic sets $(\pi^{-1}(S^j))_j$ inside which $  s^\bullet_{Z,V}$ is constant.
\end{proof}

The polynomial inequalities defining the semialgebraic sets used  in the proof of Proposition \ref{prop:semialgebraic}  are given by  \eqref{eqn:walls-finite-quiver}, \eqref{eqn:supp-lemma-refin} and \eqref{eqn:semialg-slope}. Using the result from \cite{cheng2023deterministic}, they  can be computed algorithmically. In general, this set of inequalities is not minimal and its size is not polynomial in $|P^V|$ and $\max \udim_V$.   However, one could hope to approximate $s^\theta_{Z,V}$     by  restricting it  to a fine enough regular grid function.

\bibliographystyle{abbrv}
\bibliography{main}

\end{document}

%% file: images/grid2.tex
{\scriptsize
\begin{tikzpicture}[
roundnode/.style={circle,draw,fill, inner sep=1.5pt},>=latex]

    \draw[gray,thin,opacity=0.5] (2,1)-- (8,1);
    \draw[gray,thin,opacity=0.5] (2,2)-- (8,2);
    \draw[gray,thin,opacity=0.5] (2,1)-- (2,4);
    \draw[gray,thin,opacity=0.5] (5,1)-- (5,4);
    \draw[gray,thin,opacity=0.5] (6,1)-- (6,4);

    \node[roundnode](s0) at (2,1){};
      \node[roundnode](s1) at (2,2){};
        \node[roundnode](s2) at (5,2){};
        \node[roundnode](s3) at (5,1){};
                \node[roundnode](s2) at (6,2){};
        \node[roundnode](s3) at (6,1){};

\node (labp) at (2.4,0.7){$G(p)$};
\fill[pattern=north west lines, pattern color=black,opacity=0.25] (2,1) rectangle (5,2);
\node (labcubp) at (3.5,1.5){$\mathrm{cub}_G(p)$};

\node[rectangle,inner sep=1.5pt,draw] (q) at (7.5,3.5){};
\node (labq) at (7.75,3.25) {$q$};
\node (labfloorq) at (6.75,1.7) {$G(\lfloor q\rfloor_G)$};

\end{tikzpicture}

}

%% file: images/gridmove.tex
\begin{tikzpicture}[
roundnode/.style={circle,draw, inner sep=1.5pt},>=latex]

    \draw[gray,thin,opacity=0.5] (0,1)-- (8,1);
    \draw[gray,thin,opacity=0.5] (0,2)-- (8,2);
    \draw[gray,thin,opacity=0.5] (2,0)-- (2,4);
    \draw[very thick,black] (4.5,0)-- (4.5,4);
    \draw[gray,thin,opacity=0.5] (6.25,0)-- (6.25,4);

    \draw[very thick, ->] (4.25,0.5) -- (3.75,0.5);
    \draw[very thick, ->] (4.75,0.5) -- (5.25,0.5);




\node  (a) at (2,-0.5) {$x^-$};
\node(x) at (4.5,-0.55) {$x$};
\node  (b) at (6.25,-0.5) {$x^+$};

\end{tikzpicture}